\documentclass[12pt,a4paper]{article}
\usepackage[margin=2.8cm,footskip=1cm]{geometry}
\usepackage[utf8]{inputenc}
\usepackage[T1]{fontenc}
\usepackage[english]{babel}
\usepackage{amsmath}
\usepackage{amsthm}
\usepackage{amsfonts}
\usepackage{amssymb}
\usepackage{enumitem}
\usepackage{ucs}
\usepackage{graphicx}
\DeclareGraphicsExtensions{.png,.pdf}
\usepackage[svgnames]{xcolor}
\usepackage{dsfont}
\usepackage{tikz}
\usepackage{pgfplots}
\usetikzlibrary{arrows.meta}
\pgfplotsset{compat=1.18}
\usepackage{wasysym}
\usepackage{physics}
\usepackage{mathtools}
\usepackage{xifthen}
\usepackage{url}
\usepackage{bm}
\usepackage[hidelinks]{hyperref}
\usepackage{authblk}
\usepackage[font=footnotesize]{caption}

\newcommand{\babs}[1]{\bigl\lvert #1 \bigr\rvert}
\newcommand{\Babs}[1]{\Bigl\lvert #1 \Bigr\rvert}

\DeclarePairedDelimiterX{\setof}[2]{\{}{\}}{#1 \,:\, #2}

\newcommand{\rS}{\mathfrak{s}}
\newcommand{\kerPF}{\mathcal{Z}}

\newcommand{\given}{\,|\,}
\newcommand{\bgiven}{\,\big|\,}
\newcommand{\Bgiven}{\,\Big|\,}

\newcommand{\Var}{\mathrm{Var}}
\newcommand{\Cov}{\mathrm{Cov}}

\newcommand{\BrownExc}{\mathfrak{e}}

\newcommand{\Z}{\mathbb{Z}}
\newcommand{\R}{\mathbb{R}}

\newcommand{\rmP}{\mathrm{P}}

\newcommand{\calF}{\mathcal{F}}

\newcommand{\calI}{\mathcal{I}}

\newcommand{\calM}{\mathcal{M}}
\newcommand{\calN}{\mathcal{N}}

\newcommand{\niceSlopeSet}{\mathcal{NS}}

\newcommand{\Inter}{\mathrm{Inter}}
\newcommand{\Small}{\mathrm{Small}}
\newcommand{\PreSmall}{\mathrm{PreSmall}}

\theoremstyle{plain}
\newtheorem{theorem}{Theorem}[section]
\newtheorem{lemma}[theorem]{Lemma}

\newtheorem{remark}{Remark}[section]

\newtheorem{claim}{Claim}

\theoremstyle{definition}

\newtheorem*{definition*}{Definition}

\title{Random Walk conditioned to stay above a non-flat floor: curvature effects}
\author[$\dagger$]{Sébastien Ott}
\author[$\star$]{Yvan Velenik}
\affil[$\dagger$]{EPFL, Lausanne, CH}
\affil[$\star$]{University of Geneva, Geneva, CH}

\date{\today}

\begin{document}

\maketitle

\begin{abstract}
    Let \(h:[0,1]\to\R\) be \(C^2\) and such that \(\sup_{[0,1]} h''<0\). For a (large) positive integer \(n\), set \(h_n(k) = n h(k/n)\) for any \(k\in\{0,\dots,n\}\). We consider a random walk \((S_k)_{k\geq 0}\) with i.i.d.\ centred increments having some finite exponential moments.
    We are interested in the event \(\{S\geq h_n\} = \{S_k\geq h_n(k)\;\forall k\in\{0,\dots,n\}\}\).
    It is well known that \(P(S\geq h_n \given S_0=0,\, S_n=\lceil h_n(n) \rceil) = e^{n\int_0^1 I(h'(s)) \,ds + o(n)}\), where \(I\) is the Legendre--Fenchel transform of the log-moment generating function associated to the increments. We first prove that the leading correction is of order \(e^{-\Theta(n^{1/3})}\). We then turn our attention to the conditional random walk measure \(P^h_n = P(\cdot \given S\geq h_n, S_0=0, S_n=\lceil h_n(n) \rceil)\). We prove that the one-point tails are of the form
    \(\mathbb{P}_n^h (S_k \geq h_n(k) + t n^{1/3} ) = e^{-\Theta(t^{3/2})}\) for all \(t<n^\beta\) for any \(\beta\in (0,1/6)\). Moreover, we prove that, for any \(r\geq 1\), \(E_n^h((S_k-h_n(k))^r) = \Theta(n^{r/3})\) and \(\Var_{P_n^h}(S_k) = \Theta(n^{2/3})\), for all \(k\) far enough from \(0\) and \(n\). In addition, we show that \(\Cov_{P_n^h}(S_k,S_\ell) \leq e^{-O(|\ell-k|/n^{2/3})}\) for all \(k,\ell\) not too close to \(0\) and \(n\).

    Finally, we show, in a restricted Gaussian setup, that relaxing the assumptions of smoothness or strict concavity of the obstacle drastically changes the behavior.
\end{abstract}

% \tableofcontents

\section{Introduction and main results}

The study of atypical behavior of random walk trajectories has a rich history. A cornerstone result in this area is Mogulskii’s large deviation principle~\cite{Mogulskii-1976}, building on earlier work by Borovkov~\cite{Borovkov-1967}, which states that for a random walk \((S_n)_{n\geq 0}\) with i.i.d.\ increments \((X_k)_{k\geq 1}\) having a finite log-moment generating function \(H(\lambda)=\log E(e^{\lambda X_1})\), the distribution of the rescaled linearly-interpolated trajectories \([0,1]\ni t \mapsto S_n(t) = \frac1n (S_{\lfloor nt\rfloor}+(nt-\lfloor nt\rfloor)X_{\lfloor nt\rfloor+1})\) satisfies a large deviation principle in \(C([0,1])\) with rate function \(\calI(\gamma) = \int_0^1 H^*(\gamma'(t))\, dt\) for absolutely continuous \(\gamma\) (and \(+\infty\) otherwise), where \(H^*(x)=\sup_\lambda (\lambda x - H(\lambda))\) is the Legendre--Fenchel transform of \(H\). This result can be applied, for instance, to obtain the leading order for the probability that the trajectory of the random walk remains above an obstacle.

In this work, we focus on the case of a rather general random walk bridge whose endpoints lie on a strongly concave obstacle, and whose trajectory is conditioned to stay above the latter. At the large deviation level, the corresponding question for a Brownian motion or a Brownian bridge was addressed already many decades ago; see for instance \cite{Novikov-1979, GroeneBoom-1989}. Our goals are, on the one hand, to determine the order of the correction to the large deviation probability and, on the other hand, to describe the statistical properties of the trajectories. It turns out that this problem exhibits a hallmark trio of universal exponents:
\begin{itemize}
	\item transversal fluctuations scaling as \(n^{1/3}\),
	\item a longitudinal correlation length scaling as \(n^{2/3}\),
	\item \(e^{-O(\lambda^{3/2})}\) decay for the probability of reaching a height at least \(\lambda n^{1/3}\) at a given time (for \(\lambda\) growing not too fast with \(n\)).
\end{itemize}
These exponents are characteristic of an important universality class. On the one hand, this class  contains models for which these exponents results from an interplay between curvature and fluctuations:
\begin{itemize}
	\item In the context of a Brownian bridge over \([-T,T]\), conditioned to start and end at \(0\) and remain above the semicircle \([-T,T]\ni t\mapsto \sqrt{T^2-t^2}\), it was proved~\cite{Ferrari+Spohn-2005} that the average height above the obstacle scales as \(T^{1/3}\), covariances decay exponentially with a rate of order \(T^{-2/3}\) and, after centering and rescaling, the process converges locally to a stationary diffusion on \(\R_+\), now known as a Ferrari--Spohn diffusion. They also obtain the \(3/2\) exponent, at least implicitly, as it appears as the decay exponent of the Airy function.
	\item The same exponents arise in the study of planar supercritical FK percolation clusters~\cite{Alexander-2001,Alexander+Uzun-2003,Hammond-2012}: the maximal distance between the boundary of an FK cluster conditioned to be of size \(n\gg 1\) and its convex hull is of order \(n^{1/3}\) and the size of the longest facet of the convex hull is of order \(n^{2/3}\) (up to specific logarithmic corrections).
	\item Analogous results were also obtained for planar Brownian motion conditioned to enclose a large area~\cite{Hammond+Peres-2008} and for one-dimensional random walk trajectories in a quadrant under a similar conditioning~\cite{dAlimonte+Panis-2023}.
\end{itemize}
On the other hand, this universality class also contains models for which such a curvature-fluctuations mechanism is not readily apparent:
\begin{itemize}
	\item In~\cite{Abraham+Smith-1986,Hryniv+Velenik-2004, Ioffe+Shlosman+Velenik-2015}, a general class of one-dimensional random walks conditioned to stay positive and subject to an exponential area penalization were shown to lead to the same exponents, as well as the same Ferrari--Spohn scaling limit.
	\item The same exponents remains valid for suitable generalizations to systems of ordered random walks~\cite{Ioffe+Velenik+Wachtel-2018,Caputo+Ganguly-2025,Dimitrov+Serio-2025,Hegde+Kim+Serio-2025}.
	\item The same exponents appear in important problems originating from equilibrium statistical mechanics: they describe the statistical properties of a layer of unstable phase of a planar Ising model in an external field~\cite{Velenik-2004,Ganguly+Gheissari-2021,Ioffe+Ott+Shlosman+Velenik-2022}, as well as those of the level lines of a \((2+1)\)-dimensional SOS model above a wall~\cite{Caputo+Lubetzky+Martinelli+Sly+Toninelli-2016,Caddeo+Kim+Lubetzky-2024}.
\end{itemize}

Our analysis is largely based on a reduction procedure that reformulates the problem of a random walk above a concave obstacle as a problem of a (time-inhomogeneous) random walk conditioned to stay positive and subject to a (time-inhomogeneous) external potential. This provides a robust link between the two apparently different classes of problems described above.
Note that, in the special case of a Gaussian random walk (or Brownian motion), the corresponding link can be derived by a simple change of variables (or Girsanov transformation); we exploit this in the proof of Theorem~\ref{thm:xp} below.

Moreover, we extend our investigation to a broader class of concave obstacles, lacking either smoothness or strong concavity at a point. In this setting, we rigorously derive, for a class of Gaussian random walks, the scaling behavior conjectured in the physics literature~\cite{Nechaev+Polovnikov+Shlosman+Valov+Vladimirov-2019, Smith+Meerson-2019}.

We expect that the method developed in this paper will have several interesting applications to equilibrium statistical mechanics. First, to the analysis of an interface in the planar Ising model forced to remain above a concave piece of the system's boundary. Second, to the asymptotic behavior of the 2-point function of the Ising model on \(\Z^d\), \(d\geq 2\), above its critical temperature, when the two spins are located on the boundary of the system (say, with free boundary condition); if the corresponding piece of the boundary is concave, we expect a very different behavior compared to the corresponding Ornstein--Zernike asymptotics associated to two spins in the bulk of the system, or to two spins located on a convex or affine piece of the boundary.

\subsection{Notations and main objects}

The notations/definitions given here are fixed for the whole paper. We work on some abstract, fixed, probability space \((\Omega,\calF,P)\) and all our variables are defined on that space. Consider an i.i.d.\ sequence of real random variables \(X,X_1,X_2,\dots\). Suppose that there are \(\sigma,\delta>0\) such that
\begin{equation}
 \label{eq:conditions_steps}
    E(X) = 0,\quad E(X^2) = \sigma^2,\quad E(e^{\delta|X|})<\infty.
\end{equation}
Let then \(a_*<0<b_*\) be defined by
\begin{equation}
\label{eq:dual_allowed_slopes}
    a_* = \inf\{t\in \R:\, E(e^{tX})<\infty\},
    \quad
    b_* = \sup\{t\in \R:\, E(e^{tX})<\infty\}.
\end{equation}
Define the moment and cumulant generating functions:
\begin{gather*}
    M_X:(a_*,b_*)\to \R, \quad M_X(t) = E(e^{tX}),
    \\
    H_X:(a_*,b_*)\to \R, \quad H_X(t) = \ln(M_X(t)).
\end{gather*}
For readability, set \(H\equiv H_X\), \(M\equiv M_X\).

We will frequently use the notation \(f\geq g\) with \(f,g\) either real functions or vectors: the inequality sign stands for the pointwise partial order on the corresponding object.

\subsection{Results}
\label{subsec:results}

\subsubsection{Leading-order correction to the large deviation probability}
Our first result shows that, under suitable conditions, the correction to the leading large deviation asymptotics is of order \(e^{-\Theta(n^{1/3})}\).
\begin{theorem}
\label{thm:unif_curv_free_energy}
    Let \(X, X_1,X_2,\dots\) be an i.i.d.\ family of \(\Z\)-valued, irreducible, aperiodic random variables satisfying~\eqref{eq:conditions_steps}. Let \(a_*,b_*\) be defined as in~\eqref{eq:dual_allowed_slopes}, and \(a_*<a<b<b_*\). Let \(h\in C^2([0,1])\) be a non-negative concave function with \(h''<0\), \(\mathrm{Image}(h')\subset H_X'([a,b])\), and \(h(0)=0\). Define
    \begin{equation*}
        h_n(k) = nh(k/n),\quad k=0,\dots, n.
    \end{equation*}
    Then, there are \(c_+\geq c_- >0, n_0\geq 1\) such that for any \(n\geq n_0\),
    \begin{equation*}
        e^{-c_+n^{1/3}}
        \leq
        e^{n\int_{0}^1 I(h'(s))ds} P\bigl(S\geq h_n, S_n = \lceil h_n(n)\rceil\bigr)
        \leq
        e^{-c_-n^{1/3}},
    \end{equation*}
    where \(S_0=0\), \(S_k=S_{k-1}+X_k\) and \(I(x) = \sup_{\lambda} \bigl(\lambda x - H(x)\bigr)\).
\end{theorem}
\begin{proof}
    The proof is given in Section~\ref{subsec:prf:thm:unif_curv_free_energy}. The factor \(e^{n\int_{0}^1 I(h'(s))ds}\) is extracted in Lemma~\ref{lem:large_dev_extraction}. The lower bound on the correction is proved in Lemma~\ref{lem:thm_unif_curv_free_energy:LB} and the upper bound in Lemma~\ref{lem:thm_unif_curv_free_energy:UB}.
\end{proof}

\begin{remark}
	Let us explain the condition \(\mathrm{Image}( h')\subset H_X'([a,b])\) in the case of a bounded random variable (e.g.\ uniform on \(\{-1,0,1\}\)). In that case, \(-a_*=b_*=+\infty\), and \(\lim_{t\to \pm\infty} H_X'(t)\) gives the maximal/minimal allowed slopes for the increments (\(1,-1\) in the case of the uniform over \(\{-1,0,1\}\)). The condition that \(h'\) stays away from these values is necessary to avoid ``freezing'' phenomena (for example, conditioning the walk with steps uniform in \(\{-1,0,1\}\) to stay above the identity function gives a unique possible trajectory without fluctuations).
\end{remark}

Note that the large deviation part of our theorem can be obtained using Mogulskii's Theorem, and holds for absolutely continuous \(h\). As explained below, the hypotheses of \(h\) being \(C^2\) cannot be loosened, if we wish the correction to remain of order \(n^{1/3}\).

\subsubsection{Trajectorial estimates for the conditioned random walk}
Our next results describe the statistical properties of the trajectories of the random walk conditioned to stay above the obstacle. We derive in particular the exponents discussed in the introduction.

To lighten the notation, let us denote the conditional measure by \[P_{0}^{h_n} = P( \cdot \given S_0=S_n =0, S\geq h_n)\] and introduce the rescaled process \(\rS_k = n^{-1/3}(S_k - h_n(k))\).
\begin{theorem}
\label{thm:unif_curv_typ_height_fluctuations}
    Assume the same setting as in Theorem~\ref{thm:unif_curv_free_energy}.
    Let \(\beta\in (0,1/6)\).
    With the same setup as in Theorem~\ref{thm:unif_curv_free_energy}, there are \(n_0,\lambda_0,r_0,C,c_-,c_+>0\) such that, for any \(n\geq n_0\), any \(\lambda\in[\lambda_0,n^\beta]\) and any \(k\in \{n-r_0\sqrt{\lambda}n^{2/3}, \dots, n+r_0\sqrt{\lambda}n^{2/3}\}\),
    \[
        C^{-1}e^{-c_- \lambda^{3/2}}
		\leq P_{0}^{h_n}\bigl(\rS_k \geq \lambda \bigr) \leq Ce^{-c_+ \lambda^{3/2}}.
    \]
    Moreover, there exists \(c>0\) such that
    \[
        \forall \lambda>n^{\beta},\qquad
        P_{0}^{h_n}\bigl(\rS_k \geq \lambda \bigr) \leq e^{-c \lambda}.
    \]
    In particular, for any \(r\geq 0\), there exist \(K_+(r) \geq K_-(r)>0\) such that, for all \(n\geq n_0\),
    \[
        K_-(r) \leq E_0^{h_n}(\rS_k^r) \leq K_+(r)
        \qquad\text{and}\qquad
        K_-(2) \leq \Var_{P_0^{h_n}}(\rS_k) \leq K_+(2).
    \]
\end{theorem}
\begin{remark}
    Note that this universal behavior of the tails is only valid in a restricted regime. Deeper in the tails, universality is lost and the precise behavior depends on the specificities of the random walk transition probabilities. For instance, it is easy to check (see the proof of Theorem~\ref{thm:xp} for an example) that the tails stated above remain true for any \(\lambda\) in the case of a Gaussian random walk, while they are obviously incorrect for a random walk with bounded steps (since the probability of a sufficiently large deviation is identically \(0\) in that case).
\end{remark}
\begin{proof}
    The proof is given in Section~\ref{subsec:dev_proba_moments}. The lower bound on \(P_{0}^{h_n}\bigl(\rS_0 \geq \lambda \bigr)\) is proved in Lemma~\ref{lem:dev_proba_LB}. The upper bounds can be found in Lemma~\ref{lem:dev_proba_UB}.
    The bounds on the moments and variances follow easily.
    Indeed, on the one hand, for \(n\geq n_0\),
    \[
        E_0^{h_n}(\rS_k^r)
        \geq
        \lambda_0^r P_0^{h_n}(\rS_k\geq \lambda_0)
        \geq
        C^{-1}\lambda_0^r e^{-c_-\lambda_0^{3/2}}
        >0
    \]
    and
    \begin{align*}
        E_0^{h_n}(\rS_k^r)
        &=
        \int_{0}^{\infty} P_0^{h_n}(\rS_k^r\geq \lambda) \,d\lambda \\
        &\leq
        \lambda_0^r + \int_{\lambda_0^{r}}^{n^\beta} P_0^{h_n}(\rS_k\geq \lambda^{1/r}) \,d\lambda + \int_{n^\beta}^{\infty} P_0^{h_n}(\rS_k\geq \lambda^{1/r}) \,d\lambda \\
        &\leq
        \lambda_0^r + C\int_{\lambda_0^r}^{\infty} e^{-c_+\lambda^{3/2}} \,d\lambda + \int_{n^\beta}^{\infty} e^{-c\lambda} \,d\lambda
        < \infty.
    \end{align*}
    On the other hand,
    \[
        C^{-1}e^{-c_-(E_0^{h_n}(\rS_k) + 1)^{3/2}}
        \leq
        P_0^{h_n}(\rS_k\geq E_0^{h_n}(\rS_k) + 1)
        \leq
        \Var_{P_0^{h_n}}(\rS_k)
        \leq
        E_0^{h_n}(\rS_k^2),
    \]
    and the conclusion follows from the previous bounds for all \(n\geq n_0\).
\end{proof}

\begin{theorem}
\label{thm:Covariance_decay}
    Assume the same setting as in Theorem~\ref{thm:unif_curv_free_energy}.
    There exist \(n_0\geq 0, r_0\geq 0\), \(c,C>0\) such that, for any \(n\geq n_0\) and any \(1\leq i<j\leq n\) with \(i-j\geq r_0 n^{2/3}\),
    \begin{equation*}
        \babs{ \Cov_{P_{0}^{h_n}}(\rS_i,\rS_j) } \leq C\exp(-c\tfrac{|j-i|}{n^{2/3}}).
    \end{equation*}
\end{theorem}
\begin{proof}
    The proof is given in Section~\ref{subsec:covar_decay}.
\end{proof}

\subsubsection{Relaxing the assumptions on the obstacle}

First note that the hypotheses on the random walk, in particular the existence of exponential moments, are necessary for our results. Indeed, large deviations for subexponential tails present totally different trajectorial behavior.

\medskip
One might however wonder what changes if the assumptions on the obstacle (strong concavity, \(C^2\)-regularity) are relaxed. Our next result shows that relaxing these assumption even at a single point leads to the exponents derived above ceasing to be valid in general
(except for the correlation length; the change being local, its effect on the latter is not particularly relevant).
As we are not aiming at a complete theory in this extended setting, we only discuss the technically simpler class of Gaussian random walks, conditioned to stay above an obstacle of the form \(h:[-1,1]\to\R\), \(h(x)=1-\abs{x}^p\), where \(p\geq 1\). Observe that \(h\) is not strongly concave at \(0\) when \(p>2\), while \(h\) is not \(C^2\) at \(0\) when \(p<2\), so that this class of obstacles allows one to analyze the effect of relaxing both types of assumptions.

\begin{theorem}
\label{thm:xp}
    Let \(\beta>0\). Let \((X_k)_{k\in\Z}\) be an i.i.d.\ sequence of \(\calN(0,\beta)\) random variables and \(S\) the associated random walk. Let \(p\geq 1\), and let \(h:[-1,1]\to [0,1]\) and \(\alpha_p\geq 0\) be given by
    \[
    	h(x) = 1-\abs{x}^p,\; h_n(k) = nh(k/n), \text{ and }\alpha_p = \frac{p-1}{2p-1}.
	\]
    Let \(\{S\geq h_n\} = \{S_k\geq h_n(k)\ \forall k =-n,\dots,n\}\), \(P_{0}^{h_n} = P(\cdot \given S_{-n} = S_n = 0, S\geq h_n)\) and \(\rS_0 = n^{-\alpha_p}(S_0-n)\).
    There exist \(C,c,c_-,c_+\in (0,+\infty)\), \(n_0\geq 1\), such that, for any \(n\geq n_0\) and \(\lambda\geq 1\),
    \[
        C^{-1}e^{-c_- \lambda^{(2p-1)/p}}
        \leq P_{0}^{h_n}\bigl(\rS_0 \geq \lambda \bigr)
        \leq Ce^{-c_+ \lambda^{(2p-1)/p}}.
    \]
    In particular, for all \(p\geq 0\) and \(r\geq 0\), there exist \(K_-\) and \(K_+\) such that, for any \(n\geq n_0\),
    \begin{gather*}
        K_- \leq E_0^{h_n}\bigl(\rS_0^r\bigr) \leq K_+
        \qquad\text{and}\qquad
        K_- \leq \Var_{P_{0}^{h_n}}(\rS_0) \leq K_+.
    \end{gather*}
\end{theorem}
\begin{proof}
    The proof for the tails is given in Section~\ref{sec:Gaussian_p}. The lower bound is proved in Lemma~\ref{lem:Gaussian:height_dev_proba_LB}, the upper bound in Lemma~\ref{lem:Gaussian:height_dev_proba_UB}. The resulting bounds on the moments and variance are obtained as in the proof of Theorem~\ref{thm:unif_curv_typ_height_fluctuations} (in fact, it is even simpler here, since the upper tail works for arbitrarily large values of \(\lambda)\).
\end{proof}

\begin{figure}
    \centering
    \begin{tikzpicture}[scale=0.8]
        \begin{axis}[axis lines = left, axis line style={-{Stealth[scale=1.5]}}, y=11cm, x=1cm, xtick={0,...,15}, ytick={0,0.1,0.2,0.3,0.4,0.5},ymax=.54,xmax=15.5]
        \addplot[domain=1:15.3,samples=100,red,thick]{(x-1)/(2*x-1)};
        \addplot[domain=1:15.3,samples=2, ForestGreen, dashed]{0.5};
        \end{axis}
    \end{tikzpicture}
    \caption{The graph of \(p\mapsto \alpha_p = \frac{p-1}{2p-1}\).}
    \label{Fig:plot_alpha_p}
\end{figure}
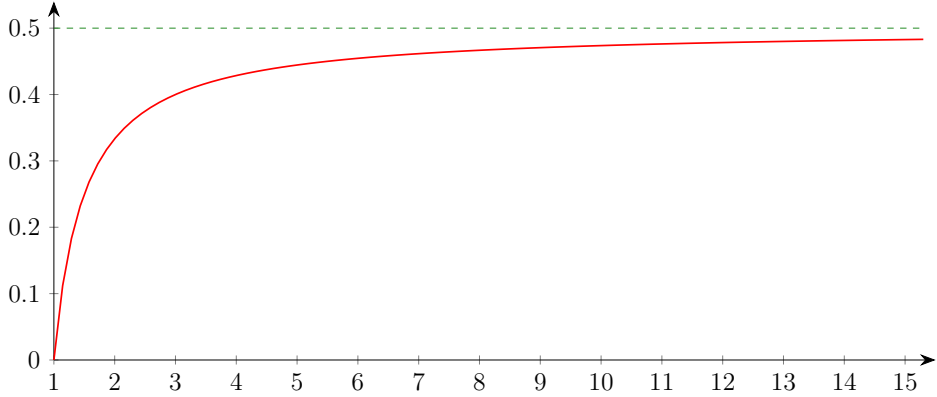

\subsection{Open problems}

\begin{itemize}
	\item In several instances (e.g., \cite{Ferrari+Spohn-2005, Ioffe+Shlosman+Velenik-2015}), it is shown that members of this universality class admit a Ferrari-Spohn (FS) diffusion as scaling limit. It would be very interesting to determine the scaling limit in the setting of this paper. One difficulty is that the time-scale at which relaxation to the FS diffusion occurs is of order \(n^{2/3}\), which coincides with the distance over which the curvature of the obstacle changes significantly, and the latter should affect the parameters of the FS diffusion. One would therefore probably need a better quantitative understanding of the relaxation than obtained in previous works, in order to address this problem.
	\item We only discuss the case of concave obstacle. However, the methods we use would allow to treat general (\(C^2\)) macroscopic obstacles, by working with the concave envelope. Indeed the random walk will concentrate in the latter under the conditionning. In particular, this would prove that the fluctuations are still of order \(n^{1/3}\) at any point where the obstacle has a strictly negative curvature, while they are of order \(n^{1/2}\) along affine pieces of the envelope.
	\item The methods developed in this paper can also be used to study a random walk bridge conditioned to stay above a \emph{mesoscopic} obstacle, say given by a function \(n^\eta h(x/n)\) for \(\eta\in (\tfrac12,1)\). The corresponding analysis in the case \(\eta<\tfrac12\) follows from the results in~\cite{Sloothaak+Wachtel+Zwart-2018}.
	\item An important extension in view of applications to spin systems (see below) would be to replace the space-time trajectory of a one-dimensional random walk, as studied here, by the spatial trajectory of a directed or massive random walk in \(Z^d\), \(d\geq 2\).
	\item One might also analyse the corresponding problem for a higher-dimensional effective interface model, for instance a GFF on \(\Z^d\) conditioned to stay above a concave obstacle.
\end{itemize}

\subsection{Acknowledgments}
We are grateful to Francesco Caravenna and Vitali Wachtel for pointers to the literature, and to Kamil Khettabi for his feedback on an earlier version of this work.
Y.V. is partially supported by the Swiss NSF through the NCCR SwissMAP.

\section{Large deviation and RW in a potential}
\label{sec:LD_part}

The first step of the proof of Theorems~\ref{thm:unif_curv_free_energy} and~\ref{thm:unif_curv_typ_height_fluctuations} is to perform a suitable change of measure which will take care of the ``large deviation part'' of the measure, and leave us with the study of a random walk in a potential. This part is not limited to \(\Z\)-valued random variables. For this whole section, let \(a_*<a<0<b<b_*\) be fixed, and write
\begin{equation*}
    \niceSlopeSet = H'\bigl([a,b]\bigr),
\end{equation*}
the set of \emph{nice slopes}.

\subsection{Output of the section}
\label{subsec:measure_tilt_LD}

Recall that \(h\in C^2([0,1],\R)\) is a concave function with strictly negative second derivative, such that \(h'\) takes values in \(\niceSlopeSet\), and \(h(0)=0\). Define
\begin{equation}
\label{eq:def:h_n}
    h_n(k) = nh(k/n),\quad k=0,\dots, n,
\end{equation}
and
\begin{equation*}
    \delta_k \equiv \delta_{n,k} = h_n(k)-h_{n}(k-1) = h'(k/n) + R_n^1(k/n),
\end{equation*}
with \(\abs{R_n^1(x)}\leq c/n\) with \(c\) uniform over \(x,n\) (Taylor--Lagrange at order \(1\), \(h''\) is uniformly bounded over \([0,1]\)). Therefore, the numbers \(\gamma_k\equiv \gamma_{n,k}\) defined via \(H'(\gamma_k) = \delta_k\) are well defined.

Let \(Y_1',Y_2',\dots\) be an independent sequence such that
\begin{equation}
\label{eq:def:tilted_RV}
    P(Y_k'\in dx) = \frac{e^{\gamma_k x}}{E(e^{\gamma_k X})} P(X\in dx).
\end{equation}
Let \(Y_k = Y'_k-E(Y_k')\), and let \(Z_0=0\), \(Z_{k+1}=Z_k+Y_{k+1}\) be the corresponding random walk. From the choice of \(\gamma_k\), one has
\begin{equation*}
    E(Y_k') = \delta_k.
\end{equation*}

The goal of this section is to obtain a representation of our problem as a ``large deviation'' part times a ``random walk in a potential'' part.
\begin{lemma}
    \label{lem:large_dev_extraction}
    Let \(\epsilon>0\). With the notations above, for any \(n\geq 1\), and non-negative measurable function \(f:\R^{n+1}\to \R_+\),
    \begin{multline}
        E\bigl(f(S) \mathds{1}_{S\geq h_n}\mathds{1}_{S_n\leq h_n(n)+\epsilon}\bigr)
        \\
        =
        e^{-n\int_{0}^1 I(h'(s))ds} E\Bigl(f(Z+h_n)\mathds{1}_{Z\geq 0}\mathds{1}_{Z_n\leq \epsilon}\exp\Bigl(-\sum_{k=1}^{n-1} \tfrac{\alpha_k}{n} Z_k\Bigr) \Bigr)e^{O_n(1)},
    \end{multline}
    with \(O_n(1)\) uniform over \(f\), and
    where
    \begin{equation}
        \alpha_k \equiv \alpha_{n,k} = -\frac{h''(k/n)}{H''((H')^{-1}(h'(k/n)))} + b_{n,k},
    \end{equation}
    and the \(b_k\equiv b_{n,k}\)'s are such that \(\sup_{k} \abs{b_{k}} = o_n(1)\). Moreover, the \(b_{n,k}\)'s are independent of \(\epsilon\).
\end{lemma}
The proof of this Lemma is the content of Section~\ref{ssec:ChangeOfMeasure}.

\begin{remark}
Observe that this reduction provides a robust link between the random walk over a concave obstacle, which is a member of the universality class in which the exponents result from an interplay between curvature and fluctuations, and a random walk conditioned to stay positive and subject to an external potential, which is very similar to the other members of the universality class, for which the mechanism is rather of the type energy \emph{vs.}\ entropy (in particular, the models investigated in~\cite{Abraham+Smith-1986, Hryniv+Velenik-2004, Ioffe+Shlosman+Velenik-2015}).
\end {remark}

\subsection{Change of measure, large deviation, and effective potential}
\label{ssec:ChangeOfMeasure}
\begin{claim}
    \label{claim:change_of_measure}
    For any measurable \(g:\R^n \to \R\),
    \begin{equation*}
        E\bigl(g(X_1,\dots, X_n)\bigr)
        \\
        =
        e^{\sum_{k=1}^n H(\gamma_k)-\gamma_k\delta_k} E\Bigl(g(Y_1+E(Y_1'),\dots,Y_n+E(Y_n'))e^{-\sum_{k=1}^n \gamma_k Y_k}\Bigr).
    \end{equation*}In particular, for any measurable \(f:\R^{n+1}\to \R\),
    \begin{equation}
    \label{eq:change_of_measure_S}
        E\bigl(f(S)\mathds{1}_{S\geq h_n}\bigr) = e^{\sum_{k=1}^n H(\gamma_k)-\gamma_k\delta_k} E\Bigl(f(Z+h_n)\mathds{1}_{Z\geq 0} e^{-\sum_{k=1}^n \gamma_k Y_k}\Bigr).
    \end{equation}
\end{claim}
\begin{proof}
    The first formula is direct (using \(E(Y'_k) = \delta_k\)). The second formula is a straightforward application of the first.
\end{proof}

We now turn to the evaluation of the sum \(\sum_{k=1}^n H(\gamma_k)-\gamma_k\delta_k\) to obtain the large deviation part of the estimate.

\begin{claim}
    \label{claim:large_dev_part}
    One has
    \begin{equation*}
        \sum_{k=1}^n H(\gamma_k)-\gamma_k\delta_k = -n\int_{0}^1 I(h'(s))ds + O_n(1).
    \end{equation*}
\end{claim}
\begin{proof}
    As \(h\) is \(C^2\), \(\sup_{x\in [k-1,k+1]}|\delta_k - h'(x/n) | \leq c/n\) with \(c\) uniform over \(k\). Now, \(\gamma_k = (H')^{-1}(\delta_k)\) has been chosen so that \(I(\delta_k) = \gamma_k\delta_k - H(\gamma_k)\), so
    \begin{equation*}
        \sum_{k=1}^n H(\gamma_k)-\gamma_k\delta_k
        =
        -\sum_{k=1}^n I(\delta_k).
    \end{equation*}
    Moreover, \(I\) is \(C^1\) on \(\niceSlopeSet\) (as the Legendre--Fenchel transform of an analytic function), so there is \(C\geq 0\) such that
    \begin{equation*}
        \sup_{k}\sup_{x\in [k-1,k+1]} \abs{I(\delta_k) - I(h'(x/n))} \leq C/n.
    \end{equation*}
    In particular
    \begin{equation*}
        \babs{I(\delta_k)-\int_{k-1}^{k} I(h'(x/n)) dx}
        \leq
        \int_{k-1}^{k} \babs{I(\delta_k)- I(h'(x/n))} dx
        \leq
        \frac{C}{n}.
    \end{equation*}
    So,
    \begin{equation*}
        \babs{\sum_{k=1}^n I(\delta_k) - \int_{0}^n I(h'(x/n)) dx}
        \leq
        n\frac{C}{n} = C,
    \end{equation*}
    which gives the claim after a change of variable.
\end{proof}

Finally, we turn to the derivation of the effective potential induced by the change of measure.

\begin{claim}
    \label{claim:effective_potential}
    Let \(\alpha_k \equiv \alpha_{n,k} = n(\gamma_k-\gamma_{k+1})\). For any \(n\geq 1\),
    \begin{itemize}
        \item one has the following identity
        \begin{equation*}
            \sum_{k=1}^n \gamma_k Y_k = \gamma_n Z_n + \sum_{k=1}^{n-1} \frac{\alpha_k}{n} Z_k;
        \end{equation*}
        \item the \(\alpha_{n,k}\)'s satisfy that
        \begin{equation*}
            \sup_{k} \Babs{\alpha_{n,k} + \frac{h''(k/n)}{H''((H')^{-1}(h'(k/n)))}} = o_n(1).
        \end{equation*}
        In particular, there exist \(\alpha_+,\alpha_-\in (0,+\infty)\) and \(n_0\geq 1\) such that
        \begin{equation*}
            \alpha_- \leq \inf_{n\geq n_0}\inf_{k}\alpha_{n,k} \leq \sup_{n\geq n_0}\sup_{k}\alpha_{n,k} \leq \alpha_+.
        \end{equation*}
    \end{itemize}
\end{claim}
\begin{proof}
    Start by doing a summation by parts :
    \begin{multline*}
        \sum_{k=1}^n \gamma_k Y_k
        =
        \sum_{k=1}^n \gamma_k (Z_k -Z_{k-1}) = \sum_{k=1}^n \gamma_k Z_k - \sum_{k=0}^{n-1} \gamma_{k+1} Z_{k}
        \\=
        \gamma_n Z_n -\gamma_1 Z_0 + \sum_{k=1}^n (\gamma_k-\gamma_{k+1}) Z_k,
    \end{multline*}
    where we have set \(\lambda_{n+1}=0\) for later convenience.
    Recalling that \(Z_0=0\), we obtain the first identity. Then, one has \(\gamma_k = (H')^{-1}(\delta_k)\), so, using a Taylor--Lagrange expansion,
    \begin{equation*}
        \gamma_k-\gamma_{k+1} = (H')^{-1}(\delta_k) - (H')^{-1}(\delta_k+ \epsilon_k) = \frac{-1}{H''((H')^{-1}(\delta_k))} \epsilon_k + r_k
    \end{equation*}
    with \(\epsilon_k = \delta_{k+1}-\delta_k = \frac{h''(k/n)}{n} + r'_k\), and \(\max(\abs{r_k},\abs{r'_k})\leq o_n(1)/n\) with \(o_n(1)\) uniform over \(k\) (use that \(h\in C^2([0,1])\) to obtain the uniformity of the \(o_n(1)\)). Now, as \(|\delta_k -h'(k/n)| \leq c/n\), we get
    \begin{equation*}
        \sup_k \Babs{\frac{1}{H''((H')^{-1}(\delta_k))} - \frac{1}{H''((H')^{-1}(h'(k/n)))}} \leq \frac{c}{n},
    \end{equation*}
    so
    \begin{equation*}
        \gamma_k-\gamma_{k+1} = -\frac{h''(k/n)}{n H''((H')^{-1}(h'(k/n)))} + r_k''
    \end{equation*}
    with \(\sup_k \abs{r''_k}\leq o_n(1)/n\). This is the first part of the second point. The second part follows from remembering that \(0< \inf -h'' \leq \sup -h'' < +\infty\) and the same holds for \(H''\) over the allowed values of \((H')^{-1}(h'(k/n))\).
\end{proof}

We are now in position to prove Lemma~\ref{lem:large_dev_extraction}.
\begin{proof}[Proof of Lemma~\ref{lem:large_dev_extraction}]
    Let \(\epsilon>0\). Applying Claims~\ref{claim:change_of_measure},~\ref{claim:large_dev_part}, and~\ref{claim:effective_potential} (and using that \(\gamma_n \epsilon = \epsilon O_n(1)\)), one obtains that for any non-negative \(f\)
    \begin{multline*}
        E\bigl(f(S) \mathds{1}_{S\geq h_n}\mathds{1}_{S_n\leq h_n(n)+\epsilon}\bigr)
        \\
        =
        e^{-n\int_{0}^1 I(h'(s))ds} E\Bigl(f(Z+h_n)\mathds{1}_{Z\geq 0} \mathds{1}_{Z_n\leq \epsilon} e^{ - \sum_{k=1}^n \frac{\alpha_k}{n} Z_k}\Bigr)e^{O_n(1)},
    \end{multline*}
    where the \(\alpha_k\)'s are given by Claim~\ref{claim:effective_potential}, and the \(O_n\) is uniform over \(f\). This is the wanted statement.
\end{proof}

\section{Proof of Theorems~\ref{thm:unif_curv_free_energy}, \ref{thm:unif_curv_typ_height_fluctuations} and~\ref{thm:Covariance_decay}}

We now restrict to \(\Z\)-valued random walks. The main reason is the use of random walk bridges measures/densities that can be ill defined in general. The same proof works for steps supported on an interval, with a positive density with respect to Lebesgue measure on that interval.

We therefore assume in this section that \(X,X_1,X_2,\dots\) is an i.i.d. sequence, that \(X\) is supported on \(\Z\), that its satisfies~\eqref{eq:conditions_steps}, and that its law is aperiodic.
For this whole section, we fix \(h:[0,1]\to \R_+\) as in Theorem~\ref{thm:unif_curv_free_energy} and define \(h_n\) as in Section~\ref{subsec:measure_tilt_LD}.
Also, recall \(Y',Y,Z,\gamma_{k}\) defined in Section~\ref{subsec:measure_tilt_LD}. We will rely on the results about random walks bridges and excursions from~\cite{Ott+Velenik-2025a}.

\subsection{Area tilted kernel representation}
\label{subsec:kernel_rep}

For \(n\geq 1\), \(0\leq l< k\leq n\), \(x,y\in \R\), define
\begin{equation}
	\kerPF_{l,k}^{n}(x,y)
	=
	E\Bigl(\mathds{1}_{Z_k= y} \prod_{i=l+1}^k \mathds{1}_{Z_i\geq 0} e^{-\frac{\alpha_i}{n}Z_i} \Bgiven Z_l = x\Bigr)
\end{equation}where the \(\alpha_i\) are given by Lemma~\ref{lem:large_dev_extraction} (they implicitly depend on \(n\)). Recall that (see Claim~\ref{claim:effective_potential}) there are \(n_0\geq 1\), and \(\alpha_-,\alpha_+\in (0,+\infty)\) such that
\begin{equation*}
	\alpha_-
	\leq
	\inf_{n\geq n_0}\inf_{1\leq i\leq n} \alpha_i
	\leq
	\sup_{n\geq n_0}\sup_{1\leq i\leq n} \alpha_i
	\leq
	\alpha_+.
\end{equation*}

\subsection{Proof of Theorem~\ref{thm:unif_curv_free_energy}}
\label{subsec:prf:thm:unif_curv_free_energy}

We are now ready to prove Theorem~\ref{thm:unif_curv_free_energy}. Looking at Lemma~\ref{lem:large_dev_extraction}, the lower bound follows from the next Lemma.
\begin{lemma}
	\label{lem:thm_unif_curv_free_energy:LB}
	There are \(n_0\geq 1\), \(c>0\) such that for any \(n\geq n_0\),
	\begin{equation*}
		\kerPF_{0,n}^{n}(0,z_n)
		\geq
		e^{-cn^{1/3}},
	\end{equation*}
	where \(z_n = \lceil h_n(n)\rceil- h_n(n)\).
\end{lemma}
\begin{proof}
	By~\cite[Lemma 7.1]{Ott+Velenik-2025a}, there is \(c>0\) such that, for all \(n\) large enough,
    \begin{equation*}
        P\bigl(Z_n= z_n, \min_{i=1,\dots,n} Z_i\geq 0, \max_{i=1,\dots,n} Z_i\leq n^{1/3} \bgiven Z_0 = 0\bigr)
        \geq
        \tfrac{C}{n} e^{-cn^{1/3}}.
    \end{equation*}
    To give the idea: the excursion has a positive probability to stay in a tube of size \(n^{1/3}\) over a time of order \(n^{2/3}\). One needs to pay this for every time interval of length \(n^{2/3}\), giving the claim. Now, if \(\max_{i=1,\dots,n} Z_i\leq n^{1/3}\),
    \begin{equation*}
        \sum_{i=1}^n \tfrac{\alpha_i}{n}Z_i
        \leq
        \alpha_+ n^{1/3},
    \end{equation*}
    and thus,
    \begin{equation*}
        \kerPF_{0,n}^{n}(0,z_n)
        \geq
        e^{-\alpha_+ n^{1/3}} P\bigl(Z_n= z_n,\, \cap_{i=1}^n \{0\leq Z_i\leq n^{1/3}\} \bgiven Z_0 = 0\bigr)
        \geq
        \tfrac{C}{n} e^{-(c+\alpha_+)n^{1/3}},
    \end{equation*}
    this is the claim.
\end{proof}

We now turn to the upper bound. We first prove a ``scale \(n^{2/3}\)'' result.
\begin{lemma}
	\label{lem:strd_time_kernel:UB}
	There are \(n_0\geq 1\), \(c>0\) such that for any \(n\geq n_0\), any \(1\leq l< k\leq n\) with \(n^{2/3} \leq k-l \leq 2n^{2/3}\), and any \(x\geq 0\),
	\begin{equation*}
		\sum_{y\geq 0}\kerPF_{l,k}^{n}(x,y)
		\leq
		e^{-c},
	\end{equation*}
	where the sum is over \(y\) such that \(P(Z_k=y\given Z_l= x)>0\).
\end{lemma}
\begin{proof}
	Let \(n\) be large enough so that \(\alpha_i \geq \alpha_-\) for all \(i\)'s. Let \(K>0\) to be fixed large later. Let \(l,k\) be as in the statement and set \(L=k-l\), \(W_0=0\), \(W_{i} = W_{i-1}+ Y_{l+i}\) (recall that \(Y_j = Z_j-Z_{j-1}\), \(i=1,\dots, n\) form an independent centred sequence). First, if \(x\leq Kn^{1/3}\), we have,
	\begin{equation*}
		\sum_{y\geq 0}\kerPF_{l,k}^{n}(x,y)
		=
		E\Bigl(\prod_{i=1}^L e^{-\alpha_{l+i}(x+W_i)/n} \mathds{1}_{x+W_i\geq 0}\Bigr)
		\leq
		P(W_L \geq -Kn^{1/3}).
	\end{equation*}
	By the (inhomogeneous) CLT, see for example~\cite[Chapter V, Theorem 3]{Petrov-1995}, that quantity is less than \(1-\frac{1}{2}P(\calN(0,1) \geq cK)< 1\) for some \(c\geq 0\) uniform over \(n\) large enough, and \(k,l\) as in the statement. Consider now \(x\geq Kn^{1/3}\). Then,
	\begin{multline*}
		\sum_{y\geq 0}\kerPF_{l,k}^{n}(x,y)
		\\
		\leq
		E\Bigl(\mathds{1}_{\max_{1\leq i \leq L}|W_i| \leq Kn^{1/3}/2}\prod_{i=1}^L e^{-\alpha_-(Kn^{1/3}+W_i)/n}\Bigr)
		+
		P(\max_{i=1,\dots,L}|W_i| > Kn^{1/3}/2)
		\\
		\leq
		e^{-\alpha_- K/2}
		+
		P(\max_{i=1,\dots,L}|W_i| > Kn^{1/3}/2)
	\end{multline*}
	where we used \(L\geq n^{2/3}\). Now, \((W_i)_{i\geq 0}\) is a martingale. So, by Doob's submartingale inequality,
	\begin{equation*}
		P(\max_{i=1,\dots,L}|W_i| > Kn^{1/3}/2) \leq \frac{4 E(W_L^2)}{K^2 n^{2/3}}.
	\end{equation*}
	But the variables \(Y_i\)'s have uniformly bounded second moment, so \(E(W_L^2)\leq cL \leq 2cn^{2/3}\) for some \(c\) depending on \(h\) and the law of the initial sequence \(X_1,X_2,\dots\) only. Taking \(K\) large enough, the last display is less than \(1/2\). This concludes the proof.
\end{proof}

We are now ready to prove the upper bound in Theorem~\ref{thm:unif_curv_free_energy}. Looking at Lemma~\ref{lem:large_dev_extraction}, the claim follows from the next Lemma.
\begin{lemma}
	\label{lem:thm_unif_curv_free_energy:UB}
	There are \(n_0\geq 1\), \(c>0\) such that for any \(n\geq n_0\),
	\begin{equation*}
		\kerPF_{0,n}^{n}(0,z_n)
		\leq
		e^{-cn^{1/3}},
	\end{equation*}
	where \(z_n = \lceil h_n(n)\rceil- h_n(n)\).
\end{lemma}
\begin{proof}
	Let \(\ell = \lfloor n^{1/3}\rfloor\). Let \(L_0=0<L_1<L_2<\dots<L_{\ell}\) be such that
	\begin{equation*}
		L_{\ell}= n,\quad n^{2/3}\leq L_{i}- L_{i-1}\leq 2n^{2/3}\quad (i=1,\dots,\ell).
	\end{equation*}
	For \(i=1,\dots, \ell\), let \(I_i = \{x\in [0,+\infty):\, P(Z_{L_i} = x \given Z_0=0) >0\}\). This is (a subset of) the intersection of \([0,+\infty)\) with some translate of \(\Z\). Then,
	\begin{equation*}
		\kerPF_{0,n}^{n}(0,z_n) 
		\leq
		\sum_{x_1\in I_1}\dots \sum_{x_{\ell}\in I_{\ell}} \prod_{i=1}^{\ell}\kerPF_{L_{i-1},L_i}^{n}(x_{i-1},x_i) 
	\end{equation*}
	where \(x_0=0\). Using Lemma~\ref{lem:strd_time_kernel:UB}, this quantity is less than \(e^{-c\ell}\) which gives the result.
\end{proof}

\subsection{Proof of Theorem~\ref{thm:unif_curv_typ_height_fluctuations}}
\label{subsec:dev_proba_moments}

Introduce \((W_{i})_{i=1}^n\) a random variable with law given by
\begin{equation*}
    E\bigl(f(W)\bigr) = \frac{1}{\kerPF_{0,n}^{n}(0,z_n)} E\Bigl(f(Z)\mathds{1}_{Z_n=z_n} \prod_{i=1}^n \mathds{1}_{Z_i\geq 0} e^{-\frac{\alpha_i}{n}Z_i} \Bgiven Z_0 = 0\Bigr).
\end{equation*}
\((W_{i})_{i=1}^n\) is a non-negative random field with the Markov property. Moreover, by Claims~\ref{claim:change_of_measure} and~\ref{claim:effective_potential}, for any measurable function \(f\)
\begin{equation*}
    E\bigl(f(S) \bgiven S_n =\lceil h_n(n) \rceil, S\geq h_n\bigr)
    =
    E\bigl(f(W+h_n)\bigr).
\end{equation*}
Introduce also the events for \(0\leq l\leq k\leq n\),
\begin{equation*}
    A^{\lambda}_{l,k} = \cap_{i=l}^k \{W_i\geq \lambda n^{1/3}\},
    \quad
    \bar{A}^{\lambda}_{l,k} = \Bigl\{\sum_{i=l}^{k} W_i\geq \lambda n^{1/3}(k-l+1)\Bigr\},
\end{equation*}
and set \(A^{\lambda}_{l,k} = A^{\lambda}_{l\vee 0,k\wedge n}\) for \(l< 0\) and/or \(k> n\). Note moreover that \(A^{\lambda}_{l,k} \subset \bar{A}^{\lambda}_{l,k}\).

\begin{lemma}
	\label{lem:high_exc_condUB}
    Let \(\beta\in(0,1/6)\). There are \(n_0,\lambda_0, r_0\geq 0\) and \(c>0\) such that for any \(n\geq n_0\), any \(\lambda_0\leq \lambda\leq n^{\beta}\), any \(1\leq k\leq n\), and any \(L\geq r_0 \sqrt{\lambda}n^{2/3}\) with \(k+L\leq n\),
    \begin{equation*}
        P\bigl( \bar{A}^{\lambda}_{k,k+L} \bgiven (A^{\lambda}_{k-L,k-1})^c, (A^{\lambda}_{k+L+1,k+2L})^c\bigr)
        \leq
        \exp(-c\tfrac{\lambda L}{n^{2/3}}).
    \end{equation*}
\end{lemma}
\begin{proof}
    We will import results from~\cite{Ott+Velenik-2025a} in several instances. They demand a certain relations between the endpoints of the walk bridge/excursion and the number of steps. In the non-trivial cases, the wanted relation will be that \(x,y\) are less than \(CL^{\alpha}\) for some \(\alpha < 2/3\) and \(C>0\) fixed. This will be the case as:
    \begin{equation*}
        x,y \leq \lambda n^{1/3} \leq CL^{\alpha}
    \end{equation*}
    if \(\lambda n^{1/3} \leq Cr_0^{\alpha} \lambda^{\alpha/2} n^{2\alpha/3}\). This last condition is fulfilled if
    \begin{equation*}
        \lambda^{\frac{2-\alpha}{2}} \leq Cr_0^{\alpha} n^{\frac{2\alpha-1}{3}}
        \impliedby n^{\beta\frac{2-\alpha}{2}} \leq Cr_0^{\alpha}n^{\frac{2\alpha-1}{3}},
    \end{equation*}
    this last condition is verified for \(n\) large enough if
    \begin{equation*}
        \beta\tfrac{2-\alpha}{2} < \tfrac{2\alpha-1}{3},
    \end{equation*}
    which has a solution \(\alpha\in (0,2/3)\) as long as \(\beta\in (0,1/6)\), which is our hypotheses.
    
    Use the shorthands \(B_{L}^- = (A_{k-L,k-1}^{\lambda})^c\) and \(B_L^+ = (A_{k+L+1,k+2L}^{\lambda})^c\). Note that under \(B_{L}^-\), there is \(t\in \{k-L,\dots, k-1\}\) such that \( W_t< \lambda n^{1/3}\), and under \(B_L^+\) there is \(t\in \{k+L+1,\dots, k+2L\}\) such that \( W_t< \lambda n^{1/3}\). Define
	\begin{gather*}
		\tau_+ = \max \bigr\{t\in \{k+L+1,\dots, k+2L\}:\, W_t<\lambda n^{1/3}\bigr\},
		\\
		\tau_- = \min \bigr\{t\in \{k-L,\dots, k-1\}:\, W_t<\lambda n^{1/3}\bigr\}.
	\end{gather*}
    Then, using Markov's property,
	\begin{align*}
		&P\bigl(\bar{A}_{k,k+L}^{\lambda} \bgiven B_{L}^-, B_{L}^+\bigr)
		\\
        &\quad=
		\sum_{s,t}\sum_{x,y} P\bigl(W_{s}=x, \tau_- = s, W_{t} = y, \tau_+ = t \bgiven B_{L}^-, B_{L}^+\bigr) P\bigl(\bar{A}_{k,k+L}^{\lambda} \bgiven W_s=x, W_t = y\bigr).
	\end{align*}
	For \(k-L\leq s <k\), \(k+L< t\leq k+2L\), and \(0\leq x,y< \lambda n^{1/3}\), we need to bound
	\begin{multline}
    \label{eq:prf:lem:high_exc_condUB:num_Den_decomp}
		P\bigl(\bar{A}_{k,k+L}^{\lambda} \bgiven W_s=x, W_t = y\bigr)
		\\
		=
		\frac{1}{\kerPF_{s,t}^n(x,y)} E\Bigl( \mathds{1}_{Z_t=y} \mathds{1}_{\sum_{j=k}^{k+L} Z_{j}\geq \lambda n^{1/3}L }\prod_{i= s+1}^t e^{-\frac{\alpha_i}{n} Z_i} \mathds{1}_{Z_i \geq 0} \Bgiven Z_{s} = x\Bigr).
	\end{multline}
	Now, using the constraint on the sum of the \(Z_i\)'s between \(k\) and \(k+L\), the expectation term is less than (for \(n\) large enough)
	\begin{multline}
		\label{eq:prf:lem:high_exc_condUB:num_UB}
		P\bigl(Z_t=y,\, \cap_{i=s+1}^{t} \{Z_i\geq 0\} \bgiven Z_s = x\bigr)\exp(-\alpha_- \tfrac{L  \lambda}{n^{2/3}})
        \\
        \leq
        \tfrac{C\min(y+1,\sqrt{L})\min(x+1,\sqrt{L})}{L^{3/2}}\exp(-\alpha_- \tfrac{L  \lambda}{n^{2/3}})
	\end{multline}
	where we used~\cite[Lemma 6.2]{Ott+Velenik-2025a}.
    We now need to lower bound the partition function \(\kerPF_{s,t}^n(x,y)\) for \(0\leq x,y\leq \lambda n^{1/3}\), and \(t-s> L\).
    
    \begin{figure}
    	\centering
    	\includegraphics[width=\textwidth]{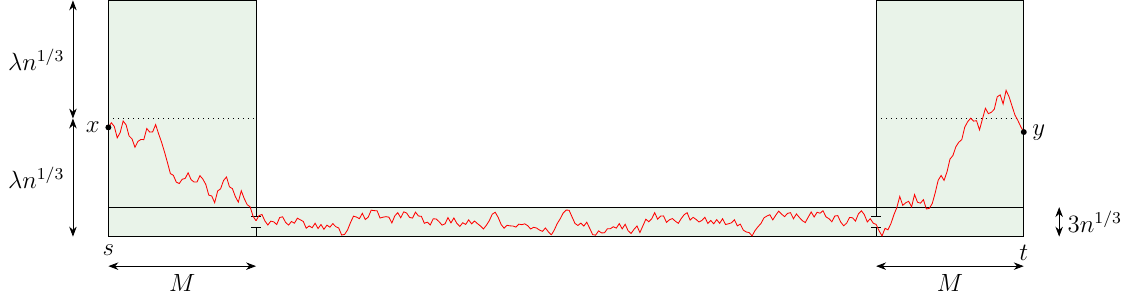}
    	\caption{Construction in the proof of Lemma~\ref{lem:high_exc_condUB}. The path must remain in the shaded region and enter the narrow tube through the openings at each extremity. The endpoints must stay below the dotted lines.}
    	\label{fig:lem:high_exc_condUB}
    \end{figure}
    Let \(M=\lfloor \frac{\alpha_-}{16 \alpha_+} L\rfloor\). By forcing \(Z\) to stay below level \(2\lambda n^{1/3}\) during the time intervals \(I_-=\{s,\dots, s+M-1\}\) and \(I_+=\{t-M+1,\dots,t\}\), to stay below level \(3n^{1/3}\) during the time interval \(I_0 = \{s+M+1,\dots, t-M-1\}\), and forcing \(n^{1/3} \leq Z_{s+M},Z_{t-M}\leq 2n^{1/3}\) (see Figure~\ref{fig:lem:high_exc_condUB}), we have the lower bound
    \begin{multline*}
		\kerPF_{s,t}^n(x,y)
		\geq
		\exp(-\tfrac{2\alpha_+\lambda}{n^{2/3}} 2 M  - \tfrac{3\alpha_+}{n^{2/3}} (t-s-2M) )
		P\Bigl(\cap_{i\in I_-\cup I_+} \{0\leq Z_i \leq 2\lambda n^{1/3}\},\\ \cap_{i\in I_0} \{0\leq Z_i \leq 3n^{1/3}\}, n^{1/3}\leq Z_{s+M},Z_{t-M}\leq 2n^{1/3}, Z_t=y \Bgiven Z_{s} = x\Bigr).
	\end{multline*}
	Now, for \(n^{1/3}\leq u,v\leq 2n^{1/3}\) admissible values for \(Z_{s+M},Z_{t-M}\) respectively, we have the following estimates.
    \begin{itemize}
        \item First, as \(\tfrac78L \leq t-M-(s+M) \leq 3L\),
        \begin{equation*}
            P\bigl(Z_{t-M} = v,\, \cap_{i\in I_0} \{0\leq Z_i \leq 3n^{1/3}\} \bgiven Z_{s+M} = u\bigr)
            \geq
            \tfrac{C}{n^{1/3}}\exp(-c\tfrac{L}{n^{2/3}})
        \end{equation*}
        for some \(c,C>0\) by~\cite[Theorem 5.5]{Ott+Velenik-2025a}.
        \item Then, there are \(C,c>0\) such that
        \begin{multline*}
            P\bigl(\cap_{i\in I_-} \{0\leq Z_i \leq 2\lambda n^{1/3}\}, Z_{s+M}=u \bgiven Z_{s} = x\bigr)
            \\
            \geq
            \tfrac{C\min(x+1,\sqrt{L})n^{1/3}}{L^{3/2}}\exp(-c\tfrac{(x-u)^2}{L})
            \geq
            \tfrac{C\min(x+1,\sqrt{L})n^{1/3}}{L^{3/2}}\exp(-c\tfrac{\lambda^2n^{2/3}}{L})
        \end{multline*}
        and
        \begin{multline*}
            P\bigl(\cap_{i\in I_+} \{0\leq Z_i \leq 2\lambda n^{1/3}\}, Z_{t} = y \bgiven Z_{t-M}=v \bigr)
            \\
            \geq
            \tfrac{C\min(y+1,\sqrt{L})n^{1/3}}{L^{3/2}}\exp(-c\tfrac{(y-v)^2}{L})
            \geq
            \tfrac{C\min(y+1,\sqrt{L})n^{1/3}}{L^{3/2}}\exp(-c\tfrac{\lambda^2n^{2/3}}{L})
        \end{multline*}
        by~\cite[Theorem 7.3]{Ott+Velenik-2025a} and the fact that \(\min(\sqrt{L},u)\geq Cn^{1/3}\) as \(\lambda\) is large enough (larger than \(1\)), and the same for \(v\).
    \end{itemize}
    Using the above observations, the fact that \(\tfrac78L \leq t-s-2M \leq 3L\), and summing over the allowed values \(n^{1/3}\leq u,v\leq 2n^{1/3}\) for \(Z_{s+M},Z_{t-M}\), we get
    \begin{align*}
		\kerPF_{s,t}^n(x,y)
		&\geq
		C\tfrac{\min(y+1,\sqrt{L})\min(x+1,\sqrt{L})n}{L^3}\exp(-2c\tfrac{\lambda^2n^{2/3}}{L}-c\tfrac{L}{n^{2/3}} -\tfrac{4\alpha_+\lambda}{n^{2/3}} M  - \tfrac{9\alpha_+}{n^{2/3}}L )
        \\
        &\geq
		C\tfrac{\min(y+1,\sqrt{L})\min(x+1,\sqrt{L})n}{L^3}\exp(-2c\tfrac{\lambda^2n^{2/3}}{L}-c\tfrac{L}{n^{2/3}} - \tfrac{\alpha_- \lambda}{4n^{2/3}} L  - \tfrac{9\alpha_+}{n^{2/3}}L )
        \\
        &\geq
		C\tfrac{\min(y+1,\sqrt{L})\min(x+1,\sqrt{L})n}{L^3}\exp(-\tfrac{\alpha_- \lambda}{2n^{2/3}} L )
	\end{align*}
	where we plugged in the value of \(M\) in the second line, and took \(\lambda_0, r_0\) large enough in the third so that for \(\lambda\geq \lambda_0\) and \(L\geq r_0\sqrt{\lambda}n^{2/3}\),
    \begin{equation*}
        c\tfrac{L}{n^{2/3}} + \tfrac{9\alpha_+}{n^{2/3}}L \leq \tfrac{\alpha_- \lambda}{8n^{2/3}} L,
        \quad
        2c\tfrac{\lambda^2n^{2/3}}{L} \leq \tfrac{\alpha_- \lambda}{8n^{2/3}} L.
    \end{equation*}
    Plugging this bound and~\eqref{eq:prf:lem:high_exc_condUB:num_UB} in~\eqref{eq:prf:lem:high_exc_condUB:num_Den_decomp}, we obtain
    \begin{equation*}
        P\bigl(A_{k,k+L}^{\lambda} \bgiven W_s=x, W_t = y\bigr)
		\leq
		\frac{CL^{3/2}}{n}\exp(-\alpha_- \tfrac{L  \lambda}{2n^{2/3}})
        \leq
		\exp(-c \tfrac{L  \lambda}{n^{2/3}}),
    \end{equation*}
    once \(\lambda_0\) is taken large enough.
\end{proof}

We then deduce from this result a similar bound but for the un-conditional measure.
\begin{lemma}
	\label{lem:high_exc_UB}
	Let \(\beta\in (0,1/6)\). There are \(n_0,\lambda_0, r_0\geq 0\), and \(c>0\) such that for any \(n\geq n_0\), any \(\lambda_0 \leq \lambda \leq n^{\beta}\), \(1\leq  k \leq n\), \(L\geq r_0\sqrt{\lambda} n^{2/3}\),
	\begin{equation*}
		P\bigl(\bar{A}_{k,k+L}^{\lambda} \bigr)
		\leq
		\exp(-c\tfrac{\lambda L}{n^{2/3}}).
	\end{equation*}
\end{lemma}
\begin{proof}
	For \(s\leq t\), introduce
	\begin{equation*}
		p(s,t) = P\bigl(\bar{A}_{s,t}^{\lambda}\bigr),
		\quad
		q(s,t) = P\bigl( \bar{A}_{s,t}^{\lambda} \bgiven (A_{2s-t,s-1}^{\lambda})^c, (A_{t+1,2t -s}^{\lambda})^c\bigr).
	\end{equation*}
	We then have the following properties that are obtained by partitioning according to the realization of \(A_{2s-t,s-1}^{\lambda}\) and/or \(A_{t+1,2t -s}^{\lambda}\), and using inclusion of events and a union bound.
    \begin{itemize}
        \item If \(2t-s< n\), and \(2s-t> 0\),
    	\begin{equation*}
    		p(s,t)
    		\leq
    		q(s,t) + p(s, 2t-s) + p(2s-t, t).
    	\end{equation*} 
        \item If \(2t-s\geq n\), and \(2s-t> 0\), \((A_{t+1,2t -s}^{\lambda})^c\) has probability \(1\) and thus
    	\begin{equation*}
    		p(s,t)
    		\leq
    		q(s,t) + p(2s-t, t).
    	\end{equation*}
        \item If \(2t-s< n\), and \(2s-t\leq 0\), \((A_{2s-t,s-1}^{\lambda})^c\) has probability \(1\) and thus
    	\begin{equation*}
    		p(s,t)
    		\leq
    		q(s,t) + p(s, 2t-s).
    	\end{equation*}
        \item Finally, if \(2t-s\geq n\), and \(2s-t\leq 0\),
        \begin{equation*}
            p(s,t)
    		=
    		q(s,t).
        \end{equation*}
    \end{itemize} Applying this and Lemma~\ref{lem:high_exc_condUB}, we get by a direct induction that
	\begin{align*}
		p(k,k+L)
		&\leq
		\exp(-c\tfrac{\lambda L}{n^{2/3}}) + p(k, k+ 2L) + p(k-L, k+L)
		\\
		&\leq
		\sum_{i=0}^{\log_2(n)} 2^{i}\exp(-c\tfrac{\lambda 2^i L}{n^{2/3}})
		\\
		&\leq
		\exp(-c\tfrac{\lambda L}{n^{2/3}})\Bigl(1- 2\exp(-c\tfrac{\lambda L}{n^{2/3}})\Bigr)^{-1}
	\end{align*}
	for \(\lambda_0\) large enough (we used \(2^i \geq i+1\) for \(i\geq 0\)).
\end{proof}

\begin{lemma}
	\label{lem:dev_proba_LB}
	Let \(\beta\in (0,1/6)\). There are \(c,n_0,\lambda_0, r_0\geq 0\) such that for any \(n\geq n_0\), any \(\lambda_0\leq \lambda\leq n^{\beta}\), and any \(r_0 \sqrt{\lambda}n^{2/3}\leq k \leq n-r_0 \sqrt{\lambda} n^{2/3}\),
    \begin{equation*}
        P\bigl( W_{k}\geq \lambda n^{1/3}\bigr)
        \geq
        e^{-c\lambda^{3/2}}.
    \end{equation*}
    In particular, \(E(W_k) \geq Cn^{1/3}\) for some \(C>0\) uniform over \(n,k\) as above.
\end{lemma}
\begin{proof}
	Take \(K,R\) large enough and let \(L= \lfloor \sqrt{\lambda} n^{2/3}\rfloor\), \(L'= \lfloor K \sqrt{\lambda} n^{2/3}\rfloor\). By Lemma~\ref{lem:high_exc_UB} (for \(K,R\) large enough),
    \begin{equation*}
        P\bigl((A_{k-L-L',k-L}^{R})^c\cap (A_{k+L,k+L+L'}^{R})^c\bigr)
		\geq
		\tfrac{1}{2}.
    \end{equation*}
    In particular, under the event \(B=(A_{k-L-L',k-L}^{R})^c\cap (A_{k+L,k+L+L'}^{R})^c\), there are \(s\in \{k-L-L',\dots, k-L\}\) and \(t\in \{k+L,\dots, k+L+L'\}\) such that \(W_{s},W_{t} \leq Rn^{1/3}\). Summing over the leftmost such \(s\), denoted \(N_-\), and the rightmost such \(t\), denoted \(N_+\), and using Markov's property, we get
    \begin{multline*}
        P\bigl( W_{k}\geq \lambda n^{1/3}\bigr)\\
        \geq
        \sum_{s= k-L-L'}^{k-L} \sum_{t= k+L}^{k+L+L'} \sum_{0\leq x \leq Rn^{1/3}} \sum_{0\leq y \leq Rn^{1/3}} P\bigl( W_{k}\geq \lambda n^{1/3} \bgiven W_{s}=x, W_{t} = y\bigr)
        \\
        \cdot P\bigl(N_- = s, N_+= t, W_s= x, W_{t} = y\bigr).
    \end{multline*}
    Now, for \(s,t,x,y\) as in the above display,
    \begin{multline*}
        P\bigl( W_{k}\geq \lambda n^{1/3} \bgiven W_{s}=x, W_{t} = y\bigr)
        \\=
        \frac{1}{\kerPF_{s,t}^n(x,y)} E\Bigl( \mathds{1}_{Z_t=y} \mathds{1}_{ Z_k \geq \lambda n^{1/3}} \prod_{i= s+1}^t e^{-\frac{\alpha_i}{n} Z_i} \mathds{1}_{Z_i \geq 0} \Bgiven Z_{s} = x\Bigr).
    \end{multline*}
    On the one hand, by~\cite[Lemma 6.2]{Ott+Velenik-2025a},
    \begin{equation*}
        \kerPF_{s,t}^n(x,y)
        \leq
        P\bigl(Z_t= y, \min_{i=s+1,\dots,t} Z_i\geq 0 \bgiven Z_s = x\bigr)
        \leq
        \tfrac{C\min(y+1,\sqrt{L})\min(x+1,\sqrt{L})}{L^{3/2}}.
    \end{equation*}
    On the other hand,
    \begin{multline*}
        E\Bigl( \mathds{1}_{Z_t=y} \mathds{1}_{Z_k \geq \lambda n^{1/3}} \prod_{i= s+1}^t e^{-\frac{\alpha_i}{n} Z_i} \mathds{1}_{Z_i \geq 0} \Bgiven Z_{s} = x\Bigr)
        \\
        \geq
        e^{-\frac{3\alpha_+\lambda}{n^{2/3}} (t-s)} P\bigl( Z_t=y, \tfrac{Z_k}{\lambda n^{1/3}} \in [1,2],\, \cap_{i= s+1}^{t}  \{0\leq Z_i \leq 3\lambda n^{1/3}\} \bgiven Z_{s} = x\bigr).
    \end{multline*}
    Now, by~\cite[Theorem 7.3]{Ott+Velenik-2025a}, as \(L \leq k-s, t-k \leq L+L'\), and \(x,y\leq Rn^{1/3}\), taking \(\lambda\) large enough,
    \begin{align*}
        P\bigl( Z_t=y,& \tfrac{Z_k}{\lambda n^{1/3}} \in [1,2],\, \cap_{i= s+1}^{t}  \{0\leq Z_i \leq 3\lambda n^{1/3}\} \bgiven Z_{s} = x\bigr)
        \\
        &=
        \sum_{\lambda n^{1/3}\leq z\leq 2\lambda n^{1/3}}P\bigl( Z_k=z,\, \cap_{i= s+1}^{k}  \{0\leq Z_i \leq 3\lambda n^{1/3}\} \bgiven Z_{s} = x\bigr)
        \\[-3mm]
        &\hspace*{4.2cm}\qquad\cdot P\bigl( Z_t=y,\, \cap_{i= k+1}^{t}  \{0\leq Z_i \leq 3\lambda n^{1/3}\} \bgiven Z_k=z \bigr)
        \\
        &\geq
        \sum_{\lambda n^{1/3}\leq z\leq 2\lambda n^{1/3}} C\tfrac{\min(x+1,\sqrt{L})\sqrt{L}}{(L+L')^{3/2}}  \cdot \tfrac{\min(y+1,\sqrt{L}) \sqrt{L}}{ (L+L')^{3/2}}\exp(-c\tfrac{(x-z)^2 + (y-z)^2}{L})
        \\
        &\geq
        C \tfrac{\min(x+1,\sqrt{L})\min(y+1,\sqrt{L})}{K^3n} \exp(-c\lambda^{3/2}).
    \end{align*}
    Combining the four last displays, we obtain
    \begin{equation*}
        P\bigl( W_{k}\geq \lambda n^{1/3} \bgiven W_{s}=x, W_{t} = y\bigr)
        \\\geq
        \frac{C}{K^3} \lambda^{3/4} \exp(-cK\lambda^{3/2}),
    \end{equation*}
    which implies the claim.
\end{proof}

\begin{lemma}
	\label{lem:dev_proba_UB}
    Let \(\beta\in (0,1/6)\). There are \(c,n_0,\lambda_0, r_0\geq 0\) such that for any \(n\geq n_0\), any \(\lambda\geq \lambda_0\), and any \(r_0 \sqrt{\lambda}n^{2/3}\leq k \leq n-r_0 \sqrt{\lambda} n^{2/3}\),
    \begin{equation*}
        P\bigl( W_{k}\geq \lambda n^{1/3}\bigr)
        \leq
        \begin{cases}
            e^{-c\lambda^{3/2}} & \text{ if } \lambda \leq n^{\beta},
            \\
            e^{-c\lambda } & \text{ if } \lambda > n^{\beta}.
        \end{cases}
    \end{equation*}
     In particular, for any \(p\geq 0\), there is \(C_p\in (0,+\infty)\) uniform over \(n,k\) as above such that \(E(W_k^p) \leq C_pn^{p/3}\).
\end{lemma}
\begin{figure}
	\centering
	\includegraphics{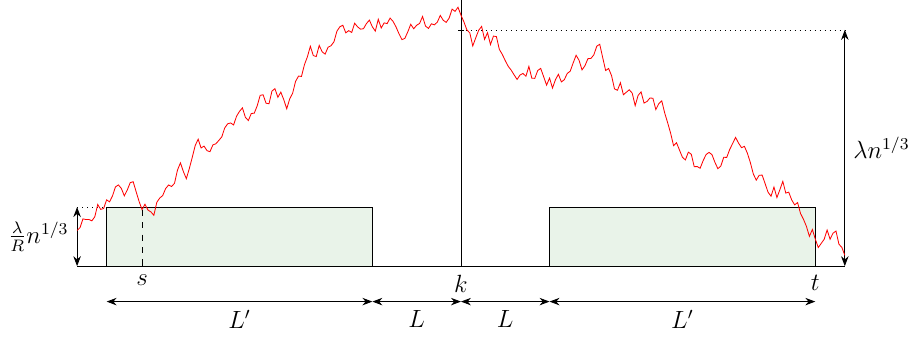}
	\caption{The construction in the proof of Lemma~\ref{lem:dev_proba_UB}. When the event \(B\) occurs, the trajectory visits both shaded rectangles. The leftmost and rightmost points at which this occurs are denoted \(s\) and \(t\) respectively.}
	\label{fig:lem:dev_proba_UB}
\end{figure}
\begin{proof}
	\textbf{We first consider \(\boldsymbol{\lambda\leq n^{\beta}}\).} Let \(K,R>0\) be large enough. Let \(L = \lceil \sqrt{\lambda} n^{2/3}\rceil\), and \(L' = \lceil K \sqrt{\lambda} n^{2/3}\rceil\). Let
    \begin{equation*}
        B = (A_{k-L-L',k-L}^{\lambda/R})^c \cap (A_{k+L,k+L+L'}^{\lambda/R})^c
    \end{equation*}
    Then,
    \begin{equation*}
        P\bigl( W_{k}\geq \lambda n^{1/3}\bigr)
        \leq
        P\bigl( W_{k}\geq \lambda n^{1/3}, B\bigr) + P(A_{k-L-L',k-L}^{\lambda/R}) + P(A_{k+L,k+L+L'}^{\lambda/R}).
    \end{equation*}
    Now, on the one hand, by Lemma~\ref{lem:high_exc_UB},
    \begin{equation*}
        P(A_{k-L-L',k-L}^{\lambda/R}) + P(A_{k+L,k+L+L'}^{\lambda/R})
        \leq
        2\exp(-c\tfrac{\lambda L'}{Rn^{2/3}})
        \leq
        2\exp(-c \tfrac{K}{R} \lambda^{3/2} ).
    \end{equation*}
    On the other hand, under the event \(B\), there are \(s\in \{k-L-L',\dots, k-L\}\) and \(t\in \{k+L,\dots, k+L+L'\}\) such that \(W_s,W_t\leq \lambda  n^{1/3}/R\). Letting \(N_-\) denote the leftmost such \(s\), and \(N_+\) the rightmost such \(t\), we have (see Figure~\ref{fig:lem:dev_proba_UB})
    \begin{multline}
    \label{eq:prf:lem:dev_proba_UB_excursion:constr_dev_Markov_decomp}
        P\bigl( W_{k}\geq \lambda n^{1/3}, B\bigr)
        =
        \sum_{s =k-L-L'}^{k-L} \sum_{t=k+L}^{k+L+L'} \sum_{0\leq x,y\leq \frac{\lambda  n^{1/3}}{R}}P\bigl( W_{k}\geq \lambda n^{1/3} \bgiven W_s =x, W_t = y\bigr)
        \\
        \cdot P\bigl( N_- = s, N_+ = t, W_s = x, W_t = y\bigr).
    \end{multline}
    Now, for \(s,t,x,y\) as in the above display,
    \begin{multline*}
        P\bigl( W_{k}\geq \lambda n^{1/3} \bgiven W_{s}=x, W_{t} = y\bigr)
        \\=
        \frac{1}{\kerPF_{s,t}^n(x,y)} E\Bigl( \mathds{1}_{Z_t=y} \mathds{1}_{ Z_k \geq \lambda n^{1/3}} \prod_{i= s+1}^t e^{-\frac{\alpha_i}{n} Z_i} \mathds{1}_{Z_i \geq 0} \Bgiven Z_{s} = x\Bigr).
    \end{multline*}
    We first lower bound the partition function:
    \begin{align*}
        \kerPF_{s,t}^n(x,y)
        &\geq
        e^{-\frac{2\alpha_+\lambda}{Rn}n^{1/3}(s-t)} P\bigl(Z_t= y,\, \cap_{i=s+1}^t \{0\leq Z_i\leq \tfrac{2\lambda n^{1/3}}{R}\}  \bgiven Z_s = x\bigr)
        \\
        &\geq
        e^{-\frac{6\alpha_+K \lambda^{3/2}}{R}} \frac{C\min(x+1,\sqrt{L})\min(y+1,\sqrt{L})}{(L+L')^{3/2}} e^{-c\tfrac{(x-y)^2}{L}}
        \\
        &\geq
        \frac{C\min(x+1,\sqrt{L})\min(y+1,\sqrt{L})}{K^{3/2}L^{3/2}} \exp(-\tfrac{6\alpha_+K \lambda^{3/2}}{R} - c\tfrac{\lambda^{3/2}}{R^2})
        \\
        &\geq
        \frac{C\min(x+1,\sqrt{L})\min(y+1,\sqrt{L})}{K^{3/2}L^{3/2}} \exp(-c \tfrac{K}{R} \lambda^{3/2}),
    \end{align*}
    where we used~\cite[Theorem 7.3]{Ott+Velenik-2025a}, the constants \(C\) and \(c\) do not depend on \(K,R\), and we used \(\lambda\geq \lambda_0\) large enough as a function of \(K,R\). Then, we upper bound the expectation term:
    \begin{align*}
        E\Bigl( \mathds{1}_{Z_t=y} \mathds{1}_{ Z_k \geq \lambda n^{1/3}} &\prod_{i= s+1}^t e^{-\frac{\alpha_i}{n} Z_i} \mathds{1}_{Z_i \geq 0} \Bgiven Z_{s} = x\Bigr)
        \\
        &\leq
        P\bigl( Z_t=y, Z_k \geq \lambda n^{1/3}, \min_{i= s,\dots, t} Z_i \geq 0\bgiven Z_{s} = x \bigr)
        \\
        &\leq
        \frac{C\sqrt{L+L'} \min(x+1, \sqrt{KL})\min(y+1, \sqrt{KL})}{\lambda n^{1/3}L^{3/2}} \exp(-c \tfrac{\lambda^2 n^{2/3}}{t-s})
        \\
        &\leq
        \frac{CK^{3/2} \min(x+1, \sqrt{L})\min(y+1, \sqrt{L})}{\lambda^{3/4}L^{3/2}} \exp(-c\lambda^{3/2}/K)
    \end{align*}
    by~\cite[Lemma 7.4]{Ott+Velenik-2025a}, where \(C\) and \(c>0\) again do not depend on \(K,R\).
    Combining the three last displays and using \(\lambda\) larger than \(\lambda_0\) large enough, we obtain
    \begin{equation*}
        P\bigl( W_{k}\geq \lambda n^{1/3} \bgiven W_{s}=x, W_{t} = y\bigr)
        \\\leq
        \tfrac{CK^3}{\lambda^{3/4}} \exp(-c(\tfrac 1K-\tfrac{c'K}{R})\lambda^{3/2})
    \end{equation*}
    where \(C\in (0,+\infty)\) and \(c,c'>0\) do not depend on \(K,R\). Plugging this into~\eqref{eq:prf:lem:dev_proba_UB_excursion:constr_dev_Markov_decomp}, we obtain the claim for any \(n\geq n_0\) and any \(\lambda\leq n^{\beta}\), once we set \(R=\tfrac12c'K^2\) and then choose \(K\) large enough.

    \medskip
    \textbf{We now derive an upper bound valid for any \(\lambda > n^{\beta}\).}
    Let \(R>0\) large and \(\epsilon>0\) small, to be fixed later. Set \(L = \lceil \epsilon\lambda n^{2/3}\rceil\). We decompose the event similarly:
    \begin{equation*}
        B = (A_{k-2L,k-L}^{R})^c\cap (A_{k+L,k+2L}^{R})^c
    \end{equation*}
    Then, by a union bound and Lemma~\ref{lem:high_exc_UB},
    \begin{equation*}
        P(W_k\geq \lambda n^{1/3})
        \leq
        2 \exp(-cR\epsilon\lambda) + P(W_k\geq \lambda n^{1/3}, B).
    \end{equation*}
    Now, letting \(N_- = \min\{k-2L \leq s \leq k-L \,:\, W_{s}\leq R n^{1/3}\}\) and \(N_+ = \max\{k+L \leq t \leq k+2L \,:\, W_t\leq R n^{1/3}\}\), we have
    \begin{multline*}
        P(W_k\geq \lambda n^{1/3}, B)
        =
        \sum_{s=k-2L}^{k-L}\sum_{t=k+L}^{k+2L} \sum_{0\leq x,y \leq R n^{1/3} } P(W_{k} \geq \lambda n^{1/3} \given W_s = x, W_t = y) 
        \\
        \cdot P(N_- = s, W_s = x, N_+ = t, W_t = y).
    \end{multline*}
    Now, we use the same expression for \(P(W_{k} \geq \lambda n^{1/3} \given W_s = x, W_t = y)\) as in the previous point. We first lower bound
    \begin{align*}
        \kerPF_{s,t}^n(x,y)
        &\geq
        e^{-\alpha_+2Rn^{1/3}L} P\bigl(Z_t= y,\, \cap_{i=s+1}^t \{0\leq Z_i\leq 2R n^{1/3}\}  \bgiven Z_s = x\bigr)
        \\
        &\geq
        e^{-2\alpha_+R\epsilon\lambda} \frac{C(x+1)(y+1)}{R^3n} e^{-c L/ (Rn^{1/3})^2}
        \\
        &\geq
        \frac{C(x+1)(y+1)}{R^3n} e^{-c \epsilon R^{-2} \lambda -2\alpha_+R\epsilon\lambda}
        \\
        &\geq
        e^{-cR\epsilon\lambda},
    \end{align*}
    by~\cite[Lemma~7.1]{Ott+Velenik-2025a}, where the constants \(C\) and \(c\) do not depend on \(R,\epsilon\), and we used \(R\) large enough in the last line together with \(\lambda > n^{\beta}\). Then, as \(x\leq R n^{1/3}\),
    \begin{align*}
        E\Bigl( \mathds{1}_{Z_t=y} \mathds{1}_{ Z_k \geq \lambda n^{1/3}} \prod_{i= s+1}^t e^{-\frac{\alpha_i}{n} Z_i} \mathds{1}_{Z_i \geq 0} \Bgiven Z_{s} = x\Bigr)
        &\leq
        P\bigl( Z_k \geq \lambda n^{1/3} \bgiven Z_{s} = x \bigr) \\
        &\leq
        P\bigl( Z_k \geq \tfrac12 \lambda n^{1/3} \bgiven Z_{s} = 0 \bigr) \\
        &\leq
        e^{-c \lambda^2n^{2/3}/L}
        \\
        &\leq 
        e^{-c \lambda/\epsilon},
    \end{align*}
    by~\cite[Lemma 2.1]{Ott+Velenik-2025a}. Combining as in the previous point and taking \(R\) large enough and then \(\epsilon\) small enough yields the desired result.
\end{proof}

\subsection{Proof of Theorem~\ref{thm:Covariance_decay}}
\label{subsec:covar_decay}

\begin{lemma}
    \label{lem:covariances_UB}
    There are \(n_0\geq 0, r_0\geq 0\), \(c,C>0\) such that for any \(n\geq n_0\), and any \(1\leq i<j\leq n\) with \(j-i\geq r_0 n^{2/3}\),
    \begin{equation*}
        \babs{ \Cov(n^{-1/3}W_i,n^{-1/3}W_j) } \leq C\exp(-c\tfrac{|j-i|}{n^{2/3}}).
    \end{equation*}
\end{lemma}
\begin{proof}
	Let \(W'\) be an independent copy of \(W\). One then has that by Markov's property,
	\begin{equation*}
		2\Cov(W_i,W_j)
		=
		E\big((W_i-W_i')(W_j-W_j')\big)
		=
		E\Big((W_i-W_i')(W_j-W_j')\prod_{k=i}^j \mathds{1}_{W_k\neq W_{k}'}\Big).
	\end{equation*}Thus, by Cauchy--Schwartz inequality, and Lemma~\ref{lem:dev_proba_UB},
	\begin{equation}
		\label{eq:prf:lem:covariances_UB:Cov_to_non_intersect}
		\babs{\Cov(W_i,W_j)}
		\leq
		Cn^{2/3} P(W_k\neq W_k'\ k=i,\dots,j)^{1/2}.
	\end{equation}

    Let \(R>0\) to be fixed later. Let \(\ell = \lfloor \frac{j-i}{3n^{2/3}}\rfloor\). Then, let \(i=L_0<L_1<\dots <L_{\ell} = j\) be such that
    \begin{equation*}
        3n^{2/3} \leq L_{k}-L_{k-1} \leq 6n^{2/3}.
    \end{equation*}
    For \(k=1,\dots, \ell\), introduce
    \begin{equation*}
        \Inter_k = \sum_{l=L_{k-1} +1}^{L_k} \mathds{1}_{W_l=W_l'},
        \quad
        \Small_k = \bigcap_{l=L_{k-1}+1}^{L_k}\{W_l\leq R n^{1/3}\}\cap \{W_l'\leq R n^{1/3}\},
    \end{equation*}
    the number of intersections of \(W\) and \(W'\) in the time interval \(\{L_{k-1}+1,\dots, L_k\}\), and the event that \(W,W'\) are both below height \(Rn^{1/3}\) in this time interval. Let then
    \begin{equation*}
        M= \sum_{k=1}^{\ell} \mathds{1}_{\Small_k}.
    \end{equation*}
    
    We will proceed in two steps. First we show that there are many places where the fields are small: for \(R\) large enough, there are \(\epsilon, c, C>0\) such that
    \begin{equation}
        \label{eq:prf:lem:covariances_UB:small}
        P\bigl(M \leq \epsilon \ell \bigr) \leq C\exp(-c\ell).
    \end{equation}
    Then, we show that when the fields are small, there is a positive probability that they meet: there is \(c>0\) such that for any \(0\leq u,v,u',v'\leq Rn^{1/3}\), and any \(k\in\{1,\dots, \ell\}\),
    \begin{equation}
        \label{eq:prf:lem:covariances_UB:intersections_when_small}
        P\bigl(\Inter_k = 0 \bgiven \Small_k, W_{L_{k-1}+1} = u, W_{L_k} = v, W_{L_{k-1}+1}' = u', W_{L_k}' = v' \bigr) \leq e^{-c}.
    \end{equation}
    From these, a comparison with i.i.d.\ Bernoulli random variables gives
    \begin{equation*}
        P\bigl(\cap_{k=i}^{j}\{W_k\neq W_k'\} \bigr)
        \leq
        Ce^{-c\ell},
    \end{equation*}
    which we plug in~\eqref{eq:prf:lem:covariances_UB:Cov_to_non_intersect} to get
    \begin{equation*}
        \babs{\Cov(W_i,W_j)}
        \leq
        Cn^{2/3}e^{-c\ell},
    \end{equation*}
    which gives the claim. Remains to prove~\eqref{eq:prf:lem:covariances_UB:small} and~\eqref{eq:prf:lem:covariances_UB:intersections_when_small}.

    \medskip
    \begin{figure}
    	\centering
    	\includegraphics{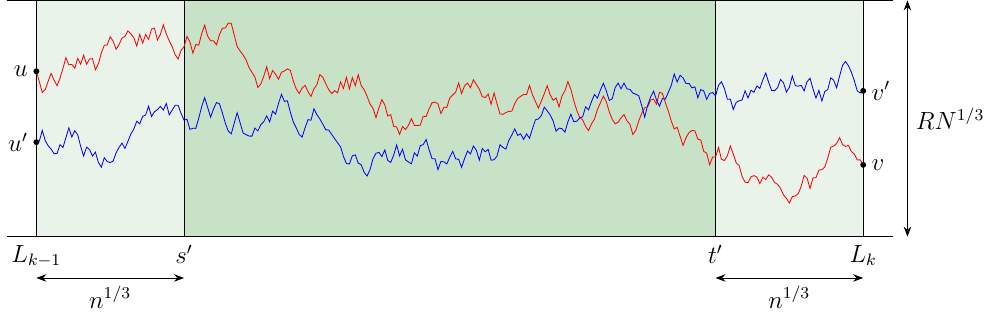}
    	\caption{The construction in the proof of Lemma~\ref{lem:covariances_UB}. The two independent paths are constrained to lie inside the shaded area. the goal is to prove that there is a reasonable probability that they intersect inside the dark shaded region.}
    	\label{fig:lem:covariances_UB}
    \end{figure}
    \textbf{We start by showing~\eqref{eq:prf:lem:covariances_UB:intersections_when_small}.} For \(0\leq u,v,u',v'\leq Rn^{1/3}\), and \(k\in\{1,\dots, \ell\}\), we have (see Figure~\ref{fig:lem:covariances_UB})
    \begin{multline*}
        P\bigl(\Inter_k > 0 \bgiven \Small_k, W_{s} = u, W_{t} = v, W_{s}' = u', W_{t}' = v' \bigr)
        \\
        =
        \frac{E\Bigl( \mathds{1}_{Z_{t}=v}\mathds{1}_{Z_{t}'=v'} (1-\prod_{l=s}^t \mathds{1}_{Z_l \neq Z_l'})\prod_{l= s}^{t} e^{-\frac{\alpha_l}{n} (Z_l+Z_l')} \mathds{1}_{Rn^{1/3}\geq Z_l, Z_l' \geq 0} \Bgiven Z_{s} = u, Z_{s}' = u'\Bigr)}{E\Bigl( \mathds{1}_{Z_{t}=v}\mathds{1}_{Z_{t}'=v'} \prod_{l= s}^{t} e^{-\frac{\alpha_l}{n} (Z_l+Z_l')} \mathds{1}_{Rn^{1/3}\geq Z_l, Z_l' \geq 0} \Bgiven Z_{s} = u, Z_{s}' = u'\Bigr)}
        \\
        \geq
        e^{-12R\alpha_+} P_{u,u'}^{v,v'}\bigl( \cup_{l= s'}^{t'}\{Z_l = Z_l'\} \bgiven D \bigr),
    \end{multline*}
    where \(Z'\) is an independent copy of \(Z\), we used the shorthands \(s = L_{k-1}+1, t= L_k\), \(s' = s+\lceil n^{2/3}\rceil, t' = t-\lceil n^{2/3}\rceil\), the fact that \(t-s \leq 6n^{2/3}\), and we introduced the double bridge measure
    \begin{equation*}
        P_{u,u'}^{v,v'}(\cdot ) = P\bigl( \cdot \bgiven Z_{s} = u, Z_{s}' = u', Z_t =v, Z_{t}' = v'\bigr),
    \end{equation*}
    and the event
    \begin{equation*}
        D = \cap_{l= s}^{t} \{Rn^{1/3}\geq Z_l, Z_l' \geq 0\}.
    \end{equation*}
    Let then
    \begin{equation*}
        N = \sum_{l=s'}^{t'} \mathds{1}_{Z_l= Z_l'}
    \end{equation*}
    be the intersection number. We will use a second moment method: \(P_{u,u'}^{v,v'}(N >0\given D) \geq \tfrac{E_{u,u'}^{v,v'}(N \given D)^2}{E_{u,u'}^{v,v'}(N^2\given D)}\). First, for \(s'\leq l\leq t'\), we have that by~\cite[Theorem 7.3]{Ott+Velenik-2025a}
    \begin{multline*}
        P_{u,u'}^{v,v'}\bigl(Z_l= Z_l' \bgiven D \bigr)
        =
        \sum_{0\leq w\leq Rn^{1/3}} P_{u,u'}^{v,v'}\bigl(Z_l= Z_l' = w \bgiven D \bigr)
        \\
        \geq
        C\sum_{0\leq w\leq Rn^{1/3}} \tfrac{\min(w+1,Rn^{1/3}-w+1, n^{1/3})^4}{n^2}e^{-cR^2}
        \geq
        \tfrac{C}{n^2} Rn^{1/3}
        n^{4/3}e^{-cR^2}
        =
        \tfrac{CR}{n^{1/3}}e^{-cR^2},
    \end{multline*}
    where the \(C\in (0,+\infty)\) do not depend on \(R\) as long as \(R \geq 6\) (and \(n\) large enough as a function of \(R\)). In particular, \(E(N\given D) \geq CRn^{1/3}\). Now, proceeding in a similar fashion as for the lower bound on the expectation of \(N\), for \(s'\leq l < r \leq t'\),
    \begin{align*}
        P_{u,u'}^{v,v'}\bigl(Z_l= Z_l', Z_r = Z_r' \bgiven D \bigr)
        &=
        \sum_{0\leq w, z\leq Rn^{1/3}} P_{u,u'}^{v,v'}\bigl(Z_l= Z_l' = w, Z_r = Z_r' = z \bgiven D \bigr)
        \\
        &\leq 
        C e^{cR^2}\sum_{0\leq w, z\leq Rn^{1/3}} \tfrac{n^{4/3}}{n^2} P(Z_{r} = z \given Z_l = w)^2
    \end{align*}
    where we used~\cite[Theorem 7.3]{Ott+Velenik-2025a} (together with straightforward algebra).
    Now, by the inhomogeneous LLT~\cite[Theorem 4.1]{Ott+Velenik-2025a}, and a straightforward large deviation bound, see for example~\cite[Lemma 2.1]{Ott+Velenik-2025a}, we have
    \begin{equation*}
        P(Z_{r} = z \given Z_l = w)
        \leq
        \begin{cases}
            \tfrac{C}{\sqrt{r-l}}e^{-c(z-w)^2/(r-l)} & \text{ if } |z-w|\leq \rho (r-l),
            \\
            e^{-c|z-w|} & \text{ otherwise,}
        \end{cases}
    \end{equation*}
    for some \(\rho,c,C >0\). Thus,
    \begin{equation*}
        \sum_{0\leq w, z\leq Rn^{1/3}} P(Z_{r} = z \given Z_l = w)^2
        \leq
        Rn^{1/3} \tfrac{C}{\sqrt{r-l}}.
    \end{equation*}
    In particular,
    \begin{equation*}
        P_{u,u'}^{v,v'}\bigl(Z_l= Z_l', Z_r = Z_r' \bgiven D \bigr)
        \leq
        \tfrac{C Re^{cR^2}}{n^{1/3}\sqrt{r-l}},
    \end{equation*}
    and so
    \begin{align*}
        E_{u,u'}^{v,v'}(N^2\given D) 
        &=
        \sum_{l=s'}^{t'} P_{u,u'}^{v,v'}\bigl(Z_l= Z_l' \bgiven D \bigr) + 2\sum_{s'\leq l < r\leq t'} P_{u,u'}^{v,v'}\bigl(Z_l= Z_l', Z_r = Z_r' \bgiven D \bigr)
        \\
        &\leq
        E_{u,u'}^{v,v'}(N\given D) + \tfrac{C Re^{cR^2}}{n^{1/3}} n^{2/3}n^{1/3}
        \\
        &\leq
        E_{u,u'}^{v,v'}(N\given D) + C Re^{cR^2} n^{2/3}.
    \end{align*}
    We thus obtained
    \begin{equation*}
        P_{u,u'}^{v,v'}(N >0\given D)
        \geq
        \frac{E_{u,u'}^{v,v'}(N \given D)^2}{E_{u,u'}^{v,v'}(N^2\given D)}
        \geq
        \frac{E_{u,u'}^{v,v'}(N \given D)}{1 + \frac{C Re^{cR^2} n^{2/3}}{E_{u,u'}^{v,v'}(N\given D)}}
        \geq
        \frac{CRn^{1/3}e^{-cR^2}}{1 + C' n^{1/3}e^{cR^2}},
    \end{equation*}
    which implies
    \begin{equation*}
        P\bigl(\Inter_k > 0 \bgiven \Small_k, W_{s} = u, W_{t} = v, W_{s}' = u', W_{t}' = v' \bigr)
        \geq
        \frac{CRe^{-cR^2}}{1 + C' e^{cR^2}}e^{-12R\alpha_+},
    \end{equation*}
    which is~\eqref{eq:prf:lem:covariances_UB:intersections_when_small}.

    \medskip
    \textbf{To conclude, we prove~\eqref{eq:prf:lem:covariances_UB:small}.}
    Introduce
    \begin{gather*}
        \overline{\Small}_k = \bigcup_{l=L_{k-1}+1}^{L_k}\{W_l\leq R n^{1/3}/10\},
        \quad
        \overline{\Small}_k' = \bigcup_{l=L_{k-1}+1}^{L_k}\{W_l'\leq R n^{1/3}/10\}
        \\
        \calM = \sum_{k=1}^{\ell} \mathds{1}_{\overline{\Small}_k},
        \quad
        \calM' = \sum_{k=1}^{\ell} \mathds{1}_{\overline{\Small}_k'},
        \\
        \PreSmall_k = \overline{\Small}_{k-1}\cap \overline{\Small}_{k+1} \cap \overline{\Small}_{k-1}'\cap \overline{\Small}_{k+1}'.
    \end{gather*}
    We will first show that \(\calM , \calM '\) are close to \(\ell\) with high probability. Then, we will show that conditionally on the realization of \(\PreSmall_k\), \(\Small_{k}\) has positive probability, which will allow to conclude. We have
    \begin{equation*}
        (\overline{\Small}_k)^c = \bigcap_{l=L_{k-1}+1}^{L_k}\{W_l> R n^{1/3}/10\}.
    \end{equation*}
    Now, under the event \(\{\calM \leq \ell - r\}\), we have that there are at least \(r\) indices \(k\in \{1,\dots,\ell\}\) such that the field is above \(R n^{1/3}/10\) in the time interval \(\{L_{k-1}+1,\dots L_k\}\). So,
    \begin{equation*}
        \{\calM \leq \ell - r\} \implies \sum_{l=i}^{j} W_l \geq \tfrac{3r R n}{10}.
    \end{equation*}
    In particular, by Lemma~\ref{lem:high_exc_UB}, for \(R\) large enough, and \(\frac{j-i}{n^{2/3}}\) large enough as a function of \(R\),
    \begin{equation*}
        P\bigl(\calM \leq \tfrac{99}{100}\ell\bigr)
        \leq
        P\Bigl(\sum_{l=i}^{j} W_l \geq \tfrac{3 R n}{1000}\ell\Bigr)
        \leq
        P\Bigl(\sum_{l=i}^{j} W_l \geq \tfrac{ R n^{1/3}}{2000}(i-j)\Bigr)
        \leq
        \exp(-\tfrac{c R}{2000}\tfrac{j-i}{n^{2/3}}).
    \end{equation*}
    Introduce now
    \begin{gather*}
        \calI_4 = \bigr\{k\in \{1,\dots,\ell-1\}:\, k = 0\mod 4\bigr\},
        \quad
        I_4 = \bigr\{k\in \calI_4:\, \mathds{1}_{\PreSmall_k} = 1\bigr\}.
    \end{gather*}
    
    Now, under \(\{\calM > \tfrac{99}{100}\ell\}\cap \{\calM' > \tfrac{99}{100}\ell\}\), one has that there are at least \(\tfrac{98}{100}\) of the indices \(k\in \{1,\dots, \ell\}\) such that \(\overline{\Small}_k\cap \overline{\Small}_k'\) is realized. In particular, there are at least \(\tfrac{96}{100}\) of the odd indices in \(\{1,\dots,\ell\}\) such that \(\overline{\Small}_k\cap \overline{\Small}_k'\) is realized. Now, this implies that there are at least \(\tfrac{92}{100}\) of the \(k\) even in \(\{2,\dots,\ell-1\}\) such that \(\PreSmall_k\) is realized, which in turn implies that \(|I_4| \geq \tfrac{84}{100} |\calI_4|\). For \(k\in I_4\), define
    \begin{gather*}
        \tau_{k}^+ = \min\{l\in \{L_{k}+1,\dots, L_{k+1}\}:\, W_{l}\leq \tfrac{R n^{1/3}}{10}\bigr\},
        \\
        \tau_k^- = \max\{l\in \{L_{k-2}+1,\dots, L_{k-1}\}:\, W_{l}\leq \tfrac{R n^{1/3}}{10}\bigr\}.
    \end{gather*}Now, for \(k\in \calI_4\), \(s\in \{L_{k-2}+1,\dots, L_{k-1}\}\), \(t\in \{L_{k}+1,\dots, L_{k+1}\}\), and \(0\leq u, v< \frac{R n^{1/3}}{10}\), we have (using the same arguments as we used in the proof of~\eqref{eq:prf:lem:covariances_UB:intersections_when_small})
    \begin{multline*}
        P\bigl( \cap_{l=L_{k-1}+1}^{L_k}\{W_k \leq Rn^{1/3}\} \bgiven W_s=u, W_t = v \bigr)
        \\
        \geq
        e^{-\alpha_+ Rn^{-2/3} (t-s)}\frac{P(Z_t = v,\, \cap_{l=s}^t\{ Rn^{1/3}\geq Z_i\geq 0\} \given Z_s = u)}{P(Z_t = v,\, \cap_{l=s}^t\{Z_i\geq 0\} \given Z_s = u)}
        \\
        \geq
        e^{-18 \alpha_+ R}\frac{C \min(u+1,n^{1/3}) \min(v+1,n^{1/3})e^{-cR^2}n}{n\min(u+1,n^{1/3}) \min(v+1,n^{1/3})}
        =
        Ce^{-18\alpha_+R}e^{-cR^2}
        \eqcolon
        q_R>0
    \end{multline*}
    by~\cite[Theorem 7.3 and Lemma 6.2]{Ott+Velenik-2025a}.
    This, using Markov's property and independence, straightforwardly implies that, conditionally on \(I_4 = I\), the family \((\mathds{1}_{\Small_k})_{k\in I}\) stochastically dominates an i.i.d. family of Bernoulli random variables indexed by \(I\) with parameter \(q_R^2\). Thus, partitioning over the values of \(I_4\), we get
    \begin{equation*}
        P\bigl( M \leq \tfrac{q_R^2}{8}\ell \bigr)
        \leq
        2P\bigl(\calM\leq \tfrac{99}{100}\ell \bigr) +
        \sum_{I\subset \calI_4:\, |I|\geq \frac{21}{100}\ell} P(I_4 = I) P\Bigl(\sum_{k\in I} \xi_k \leq \tfrac{q_R^2}{8}\ell \Bigr)
    \end{equation*}
    where the \(\xi_k\)'s, \(k\in \calI_4\) are an i.i.d. collection of Bernoulli random variables with parameter \(q_R^2\). Large deviation bounds for sums of Bernoulli and the previously obtained bound on \(P\bigl(\calM\leq \tfrac{99}{100}\ell \bigr)\) directly give~\eqref{eq:prf:lem:covariances_UB:small} for \(R\) large enough.
\end{proof}

\section{Gaussian walk: proof of Theorem~\ref{thm:xp}}
\label{sec:Gaussian_p}

\subsection{Framework and notations}

We will leave \(p\) implicit in the functions \(h,h_n\). Let
\begin{gather*}
    h:[-1,1]\to [0,1],\quad h(x) = 1-\abs{x}^p,
    \\
    h_n(x) = nh(x/n).
\end{gather*}
For readability, define
\begin{equation*}
    \alpha_p = \frac{p-1}{2p-1},\quad T=\frac{1}{\beta}.
\end{equation*}

We are interested in the typical height of \(S_0\) given \(S_{-n} = 0= S_n\) and \(S_k \geq nh(k/n)\), when \(S\) is a random walk with Gaussian increments. Let \(\Lambda_{n,m} = \{-n+1,\dots, m-1\}\). Write
\begin{equation*}
    H_{n,m}(\varphi) = \frac{1}{2}\sum_{k=-n+1}^{m}(\varphi_{k}-\varphi_{k-1})^2.
\end{equation*}
For \(g\in (\R\cup \{- \infty\})^{\Z}\), and \(\xi\in \R^{\Z}\) with \(\xi\geq g\), define the measures \(\mu_{n,m}^{\xi,g}\) via
\begin{gather*}
    d\mu_{n,m}^{\xi,g}(\varphi) = \frac{1}{Z_{n,m}^{\xi,g}}\delta(\varphi_{\Lambda_{n,m}^c}-\xi_{\Lambda_{n,m}^c})e^{-T H_{n,m}(\varphi)} \mathds{1}_{\varphi_{\Lambda_{n,m}}\geq g_{\Lambda_{n,m}}}d\varphi_{\Lambda_{n,m}},
    \\
    Z_{n,m}^{\xi,g}=\int_{\R^{\Lambda_{n,m}}} d\varphi e^{-T H_{n,m}(\varphi)}\mathds{1}_{\varphi_{\Lambda_{n,m}}\geq g_{\Lambda_{n,m}}},\quad \varphi_{-n}= \xi_{-n}, \varphi_{m} = \xi_m,
\end{gather*}
the Gaussian Free Field on \(\Lambda_{n,m}\) with boundary conditions \(\xi\) conditioned to stay above \(g\). When \(m\) is omitted, it is set to be \(n\).
From the definition, one has that the random walk \(S\) in Theorem~\ref{thm:xp} has law \(\mu_{n}^{0,h_n}\).

The measures \(\mu_{n,m}^{\xi,g}\) are strong FKG. In particular, if \(g\geq g'\), \(\xi \geq \xi'\), \(\xi\geq g\), \(\xi'\geq g'\), one has
\begin{equation*}
    \mu_{n,m}^{\xi',g'} \preccurlyeq \mu_{n,m}^{\xi,g}.
\end{equation*}
We will refer to this property as simply ``monotonicity''. When \(g\) is omitted, it is set to be constant \(-\infty\) (random walk bridge/Gaussian Free Field).

\subsection{Lower bound on height deviation probability}
\label{subsec:Gaussian:LowerBoundHeightdeviation}

\begin{figure}
	\centering
    \includegraphics{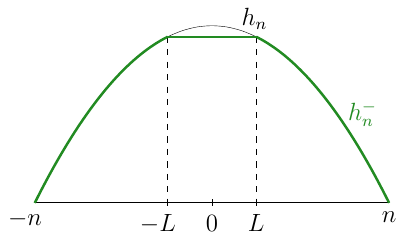}
    \hspace*{5mm}
    \includegraphics{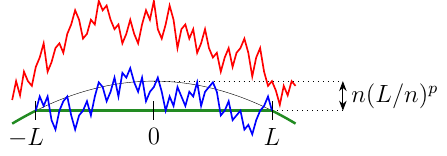}
	\caption{The construction in Section~\ref{subsec:Gaussian:LowerBoundHeightdeviation}. Left: the truncated obstacle \(h_n^-\) (green). Right: zoom near the plateau. The trajectory in red is the original process, conditioned to stay above \(h_n\) (black); the trajectory in blue is the random walk bridge that is stochastically dominated by the original process.}
\end{figure}

For \(n\geq L\geq 0\), introduce
\begin{equation*}
    h_n^-(k) = \begin{cases}
        h_n(L) & \text{ if } |k|\leq L\\
        h_n(k) & \text{ else }
    \end{cases}.
\end{equation*}
One has \(h_n^-\leq h_n\) so
\begin{equation*}
    \mu_{n}^{0,h_n^-} \preccurlyeq \mu_{n}^{0,h_n}.
\end{equation*}

\begin{claim}
    \label{claim:p_obstacle_Gaussian_LB_hminus}
    For any \(L\geq 1\), and any \(\lambda'\geq 0\),
    \begin{equation*}
        \mu_n^{0,h_n^-}\bigl(\varphi_0 \geq h_n(L) + \lambda'\sqrt{L}\bigr) \geq 2P\bigl(\calN(0,1)\geq \lambda'\sqrt{2/\beta}\bigr).
    \end{equation*}
\end{claim}
\begin{proof}
    Let
    \begin{equation*}
        g_0(k) =
        \begin{cases}
            0 & \text{ if } k=0,\\
            -\infty &\text{ if } k\neq 0.
        \end{cases}
    \end{equation*}
    By monotonicity (strong FKG) of \(\mu\), and the spatial Markov property, one has
    \begin{multline*}
        \mu_n^{0,h_n^-}\bigl(\varphi_0 \geq h_n(L) + \lambda'\sqrt{L}\bigr)
        \geq
        \mu_{L}^{h_n(L),h_n(L)}\bigl(\varphi_0 \geq h_n(L) + \lambda'\sqrt{L}\bigr)
        \\
        =
        \mu_{L}^{0,0}\bigl(\varphi_0 \geq \lambda'\sqrt{L}\bigr)
        \geq
        \mu_{L}^{0,g_0}\bigl(\varphi_0 \geq \lambda'\sqrt{L}\bigr).
    \end{multline*}
    Now, the last probability can be rephrased as
    \begin{equation*}
        \mu_{L}^{0,g_0}\bigl(\varphi_0 \geq \lambda'\sqrt{L}\bigr) = \frac{\mu_{L}^{0}\bigl(\varphi_0 \geq \lambda'\sqrt{L}\bigr)}{\mu_{L}^{0}\bigl(\varphi_0 \geq 0\bigr)} = 2\mu_{L}^{0}\bigl(\varphi_0 \geq \lambda'\sqrt{L}\bigr),
    \end{equation*}
    by symmetry. Under \(\mu_{L}^{0}\), \(\varphi_0\) is a centred Gaussian random variable with variance \(\beta L /2\)
    by Lemma~\ref{lem:Gaussian_Walk:Walk_Bridge_are_Gaussian}. Thus,
    \begin{equation*}
        \mu_{L}^{0}\bigl(\varphi_0 \geq \lambda'\sqrt{L}\bigr) = \frac{1}{\sqrt{\pi \beta L}}\int_{\lambda'\sqrt{L}}^{\infty}dx e^{-x^2/(\beta L)}= \frac{1}{\sqrt{2\pi}} \int_{\lambda'\sqrt{2/\beta}}^{\infty}dx e^{-x^2/2},
    \end{equation*}
    which gives the claim.
\end{proof}

\begin{lemma}
\label{lem:Gaussian:height_dev_proba_LB}
   Let \(p\geq 1\). Let \(S\) be the random walk with i.i.d.\ steps of law \(\calN(0,\beta)\). Then, for any \(\lambda \geq 0\),
    \begin{equation*}
        P\bigl(S_0 \geq n + \lambda n^{\alpha_p} \bgiven S_{-n}=S_n =0,\, S\geq h_n\bigr) \geq 2P\bigl(\calN(0,1)\geq c_{p,\beta} \sqrt{\lambda^{(2p-1)/p}}\bigr),
    \end{equation*}
    where \(c_{p,\beta} = 2p\sqrt{2/\beta}\). In particular, for \(\lambda\geq 1\),
    \begin{equation*}
        P\bigl(S_0 \geq n + \lambda n^{\alpha_p} \bgiven S_{-n}=S_n =0,\, S\geq h_n\bigr) \geq Ce^{-c\lambda^{(2p-1)/p}},
    \end{equation*}
    for some \(C,c>0\) depending only on \(p,\beta\).
\end{lemma}
\begin{proof}
    The law of \(S\) under \(P(\cdot \given S_{-n}=S_n =0,\, S\geq h_n)\) is \(\mu_{n}^{0,h_n}\). Let \(a>0\). Taking \(L= an^{2\alpha_p}\) and \(\lambda' = \frac{1}{\sqrt{a}}(\lambda+a^p)\) in Claim~\ref{claim:p_obstacle_Gaussian_LB_hminus}, one obtains that
    \begin{multline*}
        \mu_{n}^{0,h_n^-}\bigl(\varphi_{0}\geq n + \lambda n^{\alpha_p}\bigr)
        =
        \mu_{n}^{0,h_n^-}\bigl(\varphi_{0}\geq h_n(L) + \tfrac{L^p}{n^{p-1}} +  \tfrac{\lambda}{\sqrt{a}}\sqrt{L}\bigr)
        \\
        =
        \mu_{n}^{0,h_n^-}\bigl(\varphi_{0}\geq h_n(L) + (\tfrac{a^p}{\sqrt{a}} +  \tfrac{\lambda}{\sqrt{a}})\sqrt{L}\bigr)
        \geq
        2P\bigl(\calN(0,1)\geq \lambda'\sqrt{2 /\beta}\bigr),
    \end{multline*}
    as \(h_n(L) = n- \tfrac{L^p}{n^{p-1}} = n- a^{p-1/2}\sqrt{L} \), and \(L^{1/2} = a^{1/2}n^{\alpha_p}\).

    Taking \(a = (\lambda/(2p-1))^{1/p}\), we get
    \begin{equation*}
        \lambda'= \frac{2p}{2p-1} \lambda^{(2p-1)/2p} (2p-1)^{1/2p}
        \leq
        2p \lambda^{(2p-1)/2p}
    \end{equation*}
    and using \(\mu_{n}^{0,h_n} \succcurlyeq \mu_{n}^{0,h_n^-}\), one obtains that
    \begin{equation*}
        \mu_{n}^{0,h_n}\bigl(\varphi_{0}\geq n + \lambda n^{\alpha_p}\bigr) \geq 2P\bigl(\calN(0,1)\geq 2p \sqrt{2\lambda^{(2p-1)/p}\smash{/}\beta} \bigr),
    \end{equation*}
    which is the wanted result. The last part follows from
    \begin{equation*}
        P(\calN(0,1)\geq a)
        =
        \frac{1}{\sqrt{2\pi}}\int_{a}^{\infty} dx e^{-x^2/2}
        \geq
        \frac{1}{a(1+a^{-2})\sqrt{2\pi}}e^{-a^2/2}
    \end{equation*}
    for \(a>0\).
\end{proof}

\subsection{Upper bound on height deviation probability}
\label{subsec:Gaussian:height_dev_proba_UB}

	\begin{figure}
	\centering
    \includegraphics{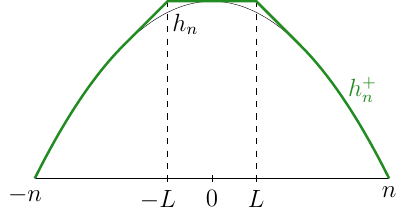}
    \hspace{5mm}
    \includegraphics{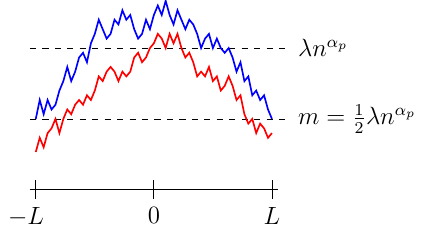}
	\caption{The construction in Section~\ref{subsec:Gaussian:height_dev_proba_UB}. Left: The extended concave obstacle \(h_n^+\) (green). Right: in red, the original process after the change of variables \(\phi\mapsto\psi\) flattening the extended obstacle; when \(\psi_{-L}\vee \psi_L \leq m\), which is likely due to the effective potential acting at these two points, this process is stochastically dominated by the bridge represented in blue.}
\end{figure}

For the upper bound, we will proceed in a similar fashion as for the lower bound, relying again on monotonicity. For \(0\leq L\leq \frac{p-1}pn\), introduce the function \(h_n^+:[-n,n]\to \R\) defined by
\begin{equation*}
    h_n^+(x) =
    \begin{cases}
        n & \text{ if } \abs{x}\leq L,
        \\
        n-\frac{p^pL^{p-1}}{(p-1)^{p-1}n^{p-1}}\bigl( \abs{x} - L\bigr) & \text{ if } L<\abs{x}\leq \tfrac{p}{p-1}L,
        \\
        h_n(x) = n- \frac{\abs{x}^p}{n^{p-1}} & \text{ else}.
    \end{cases}
\end{equation*}
Note that this function is concave and that \(h_n^+\geq h_n\). In particular, the lattice FKG property implies that
\begin{equation}
\label{eq:Gaussian:h_to_hplus}
    \mu_{n}^{0,h_n}(\varphi_0 \geq \lambda)
    \leq
    \mu_{n}^{0,h_n^+}(\varphi_0 \geq \lambda).
\end{equation}
It thus suffices to bound the last probability.
Introduce
\begin{equation*}
    \gamma_i
    =
    2h_n^+(i) - h_n^+(i-1)-h_n^+(i+1) \geq 0.
\end{equation*}
One has (recall \(T= 1/\beta\)), for \(f:\R^{\{-n,\dots,n\}}\to \R\),
\begin{align}
\label{eq:Gaussian_change_of_measure}
    \nonumber\int d\mu_n^{0,h_n^+}(\varphi) f(\varphi)
    &=
    \frac{\int d\varphi e^{-TH_{n}(\varphi) } f(\varphi)\mathds{1}_{\varphi\geq h_n^+} }{\int d\varphi e^{- TH_{n}(\varphi) } \mathds{1}_{\varphi\geq h_n^+}}
    \\
    &=
    \frac{\int d\psi e^{-TH_{n}(\psi) } e^{-T\sum_{i=1-n}^{n-1}\gamma_i \psi_i} f(\psi + h_n^+) \mathds{1}_{\psi\geq 0} }{\int d\psi e^{-TH_{n}(\psi) } e^{-T\sum_{i=1-n}^{n-1}\gamma_i \psi_i}\mathds{1}_{\psi\geq 0}},
\end{align}
where \(\varphi_{-n}=\varphi_n = 0\), and we used the change of variable \(\psi = \varphi - h_n^+\), and the identity
\begin{align}
    H_n(\psi + h_n^+)
    &=
    H_n(\psi) + H_n(h_n^+) + \sum_{i=1-n}^{n}(\psi_i-\psi_{i-1})(h_n^+(i)-h_{n}^+(i-1))
    \\
    &=
    H_n(\psi) + H_n(h_n^+) + \sum_{i=1-n}^{n-1}\psi_i \gamma_i.
\end{align}
We will use heavily that
\begin{equation*}
    \gamma_{-L} = \gamma_L = \frac{p^p L^{p-1}}{(p-1)^{p-1} n^{p-1}}.
\end{equation*}
Denote \(Q'',Q',Q\) the probability measures given by
\begin{gather*}
    dQ''(\psi) \propto  e^{-TH_{n}(\psi) } e^{-T\sum_{i=1-n}^{n-1}\gamma_i \psi_i} \mathds{1}_{\psi\geq 0}d\psi,
    \\
    dQ'(\psi) \propto  e^{-TH_{n}(\psi) } e^{-T\gamma_L \psi_L - T \gamma_{L}\psi_{-L}} \mathds{1}_{\psi\geq 0}d\psi,
    \\
    dQ(\psi)
    \propto
    g_{n,\beta}(\psi) e^{-T\gamma_L \psi_L - T \gamma_{L}\psi_{-L}} d\psi,
\end{gather*}
where
\begin{equation*}
    g_{n,\beta}(\psi)
    =
    g_{1,\dots,2n-1}^{(2n)}\bigl( \sqrt{T}\psi_{1-n},\dots, \sqrt{T}\psi_{n-1} \bigr),
\end{equation*}
with \(g_{1,\dots,2n-1}^{(2n)}\) the density of the Brownian excursion of length \(2n\) at integer times (see Appendix~\ref{app:comparison_Brown_excursion}). Now, again using the FGK-lattice property and Lemma~\ref{lem:Gaussian:excursion_comparison}, one has that
\begin{equation*}
    Q''\preccurlyeq Q'\preccurlyeq Q.
\end{equation*}
We start with the technical part of the estimate.
\begin{lemma}
    \label{lem:Gaussian:height_normal_dev_proba_UB}
    Let \(\beta>0, p\geq 1\). Then, there are \(c>0, n_0\geq 1, \lambda_0\geq 1\) such that for any \(n\geq n_0\), \(\lambda \geq \lambda_0\) with \(\lambda n^{\alpha_p}< n/2^p\),
    \begin{equation*}
        \mu_n^{0,h_n^+}(\varphi_0\geq \lambda n^{\alpha_p}+n)
        \leq
        e^{-c\lambda^{(2p-1)/p}}.
    \end{equation*}
\end{lemma}
\begin{proof}
    Introduce \(M = \max(\psi_{-L},\psi_L)\). Then, using~\eqref{eq:Gaussian_change_of_measure}, monotonicity, and Markov's property, one has that, for any \(m>0\) fixed,
    \begin{multline*}
        \mu_n^{0,h_n^+}(\varphi_0\geq \lambda n^{\alpha_p}+n)
        =
        Q''(\psi_0 \geq \lambda n^{\alpha_p})
        \leq
        Q'(\psi_0 \geq \lambda n^{\alpha_p})
        \\
        =
        Q'\bigl( \mu_{L}^{\psi,0}(\varphi_0 \geq \lambda n^{\alpha_p})\bigr)
        \leq
        Q'\bigl( \mu_{L}^{M,0}(\varphi_0 \geq \lambda n^{\alpha_p})\bigr)
        \leq
        \mu_{L}^{m,0}(\varphi_0 \geq \lambda n^{\alpha_p}) + Q'\bigl( M\geq m\bigr).
    \end{multline*}
    We now make a suitable choice of \(L\) and \(m\) as functions of \(\lambda\): fix \(L\) integer and \(m\) real to satisfy (taking \(n,\lambda\) large enough depending on \(p\))
    \begin{equation*}
        \tfrac{1}{2}\lambda^{1/p} n^{2\alpha_p} \leq L \leq \lambda^{1/p} n^{2\alpha_p},
        \quad
        m = \tfrac{1}{2}\lambda n^{\alpha_p}.
    \end{equation*}
    By our constraint on \(\lambda n^{\alpha_p}\), this implies
    \begin{equation*}
        L\leq \tfrac{n}{2},
        \quad
        \sqrt{L}\leq 2\lambda^{-(2p-1)/2p}m \leq \tfrac{m}{2}.
    \end{equation*}
    where we used \(\lambda\geq \lambda_0\) for a suitable choice of \(\lambda_0\) in the second part.
    We first bound \(\mu_{L}^{m,0}(\varphi_0 \geq \lambda n^{\alpha_p})\). One has that using monotonicity once again,
    \begin{multline*}
        \mu_{L}^{m,0}(\varphi_0 \geq \lambda n^{\alpha_p} )
        =
        \mu_{L}^{0,-m}(\varphi_0 \geq \lambda n^{\alpha_p}/2)
        \leq
        \mu_{L}^{0,0}(\varphi_0 \geq \lambda n^{\alpha_p}/2)
        \\
        \leq
        c\lambda^{1-1/2p}e^{-\lambda^{2-1/p}/8\beta}
        =
        c\lambda^{(2p-1)/2p}e^{-\frac{\lambda^{(2p-1)/p}}{8\beta}}
        \leq
        c'e^{-\frac{\lambda^{(2p-1)/p}}{9\beta}},
    \end{multline*}
    where \(c,c'>0\) are constant depending only on \(p,\beta\), and we used Lemma~\ref{lem:Gaussian:excursion_marginals} in the second inequality, the choices of \(L,m\), and \(\lambda \geq 1\). We then bound \(Q'\bigl( M\geq m\bigr)\).
    We first use stochastic ordering: as \(Q'\preccurlyeq Q\),
    \begin{equation*}
        Q'(M\geq m) \leq Q(M\geq m) \leq
        Q\bigl(\psi_L+\psi_{-L}\geq m\bigr).
    \end{equation*}
    Then, (undefined integrals are over \(\R^2\))
    \begin{align*}
        Q\bigl(\psi_L+\psi_{-L}\geq m\bigr)
        &=
        \frac{\int dx dy \, \mathds{1}_{x+y\geq m} xy \, e^{-\frac{x^2+y^2}{2(n-L)}}e^{-T\gamma_L (x+y)}
        \bigl(e^{-\frac{(x-y)^2}{4L}} - e^{-\frac{(x +y)^2}{4L}}\bigr)
        \mathds{1}_{x,y>0}}{\int dx dy \, xy \, e^{-\frac{x^2+y^2}{2(n-L)}}e^{-T\gamma_L (x+y)}
        \bigl(e^{-\frac{(x-y)^2}{4L}} - e^{-\frac{(x +y)^2}{4L}}\bigr)
        \mathds{1}_{x,y>0}}
        \\
        &=
        \frac{\int_m^{\infty} ds \int_{-s}^{s} dt (s^2- t^2)e^{-\frac{s^2+t^2}{4(n-L)}}e^{-T\gamma_L s}
        e^{-\frac{t^2}{4L}}\bigl(1 - e^{-\frac{s^2-t^2}{4L}}\bigr)}{\int_{0}^{\infty} ds \int_{-s}^s dt (s^2-t^2)e^{-\frac{s^2+t^2}{4(n-L)}}e^{-T\gamma_L s}
        e^{-\frac{t^2}{4L}}\bigl(1 - e^{-\frac{s^2-t^2}{4L}}\bigr)}
        \\
        &\leq
        \frac{\int_m^{\infty} ds \int_{-s}^{s} dt \, s^2e^{-\frac{s^2+t^2}{4(n-L)}}e^{-T\gamma_L s}
        e^{-\frac{t^2}{4L}}}{\int_{\sqrt{L}}^{\infty} ds \int_{-s/2}^{s/2} dt (s^2-t^2)e^{-\frac{s^2+t^2}{4(n-L)}}e^{-T\gamma_L s}
        e^{-\frac{t^2}{4L}}\bigl(1 - e^{-\frac{s^2-t^2}{4L}}\bigr)}
        \\
        &\leq
        \frac{4\int_m^{\infty} ds \, s^2 e^{-\frac{s^2}{4(n-L)}}e^{-T\gamma_L s} \int_{-s}^{s} dt \, e^{-\frac{t^2}{4(n-L)}}
        e^{-\frac{t^2}{4L}}}{3\int_{\sqrt{L}}^{\infty} ds \, s^2 e^{-\frac{s^2}{4(n-L)}}e^{-T\gamma_L s} \int_{-s/2}^{s/2} dt \, e^{-\frac{t^2}{4(n-L)}}
        e^{-\frac{t^2}{4L}}\bigl(1 - e^{-\frac{3}{16}}\bigr)},
    \end{align*}
    where we changed variables to \(s= x+y\), \(t=x-y\) in the second line, and we used monotonicity of the integrand in the denominator, as well as the restriction on the integration domain in the last line. We then have that for \(s\geq \sqrt{L}\), (recall \(L\leq n/2\))
    \begin{multline*}
        e^{-\frac{1}{8}}\sqrt{L}
        \leq
        \int_{-\sqrt{L}/2}^{\sqrt{L}/2} dt \, e^{-\frac{t^2}{2n}}
        e^{-\frac{t^2}{4L}}
        \leq
        \int_{-s/2}^{s/2} dt \, e^{-\frac{t^2}{4(n-L)}}
        e^{-\frac{t^2}{4L}}
        \\
        \leq
        \int_{-s}^{s} dt \, e^{-\frac{t^2}{4(n-L)}}
        e^{-\frac{t^2}{4L}}
        \leq
        \int_{\R} dt \, e^{-\frac{t^2}{4L}}
        =
        \sqrt{4\pi L}.
    \end{multline*}
    Plugging this estimates in the bound we had for \(Q\bigl(\psi_L+\psi_{-L}\geq m\bigr)\), we obtain that
    \begin{align*}
        Q\bigl(\psi_L+\psi_{-L}\geq m\bigr)
        &\leq
        C\frac{\int_m^{\infty} ds \, s^2 e^{-\frac{s^2}{4(n-L)}}e^{-T\gamma_L s}}{\int_{\sqrt{L}}^{\infty} ds \, s^2 e^{-\frac{s^2}{4(n-L)}}e^{-T\gamma_L s}}
        \\
        &=
        Ce^{-T\gamma_L (m-\sqrt{L})}\frac{\int_{\sqrt{L}}^{\infty} ds \, (s+m-\sqrt{L})^2 e^{-\frac{(s+m-\sqrt{L})^2}{4(n-L)}}e^{-T\gamma_L s}}{\int_{\sqrt{L}}^{\infty} ds \, s^2 e^{-\frac{s^2}{4(n-L)}}e^{-T\gamma_L s}}
        \\
        &\leq
        Ce^{-T\gamma_L m/2}(1+m/\sqrt{L})^2 \frac{\int_{\sqrt{L}}^{\infty} ds \, s^2 e^{-\frac{s^2}{4(n-L)}}e^{-T\gamma_L s}}{\int_{\sqrt{L}}^{\infty} ds \, s^2 e^{-\frac{s^2}{4(n-L)}}e^{-T\gamma_L s}},
    \end{align*}
    where \(C = (1-e^{3/16})\frac{8\sqrt{\pi}}{3}e^{1/8}\), and we used \(\sqrt{L}\leq m/2\) and monotonicity of the integrand in the numerator. Plugging in the values of \(m,L,\gamma_L\), we get
    \begin{align*}
        Q\bigl(\psi_L+\psi_{-L}\geq m\bigr)
        &\leq
        C'\lambda^{(2p-1)/p} e^{-Tc_p(\lambda^{1/p} n^{2\alpha_p}/2n)^{p-1} \lambda n^{\alpha_p}/4}
        \\
        &\leq
        C'\lambda^{(2p-1)/p} \exp\bigl(-\tfrac{Tc_p}{2^{p+1}} \lambda^{(2p-1)/p}\bigr)
    \end{align*}
    with \(c_p = \frac{p^p}{(p-1)^{p-1}}\), and \(C'\) some constant depending only on \(p\). This concludes the proof of the lemma.
\end{proof}
To conclude, we treat the ``very large'' deviation regime. For \(K\geq n/2^p\), one has
\begin{equation}
    \mu_{n}^{0,h_n}(\varphi_0 \geq n + K)
    \leq
    \mu_{n}^{n,n}(\varphi_0 \geq n + K)
    =
    \mu_{n}^{0,0}(\varphi_0 \geq K)
    \\
    \leq
    c\tfrac{K}{\sqrt{n}} e^{-\frac{K^2}{2\beta n}}
\end{equation}
where \(c>0\) is a constant depending only on \(\beta,p\), and we used monotonicity in the first inequality, and Lemma~\ref{lem:Gaussian:excursion_marginals} in the last. In particular, if one writes \(K=\lambda n^{\alpha_p}\), one has first that \(K/\sqrt{n} \leq \lambda\) and, since \(\lambda\geq 2^{-p}n^{1-\alpha_p}\) and \(\alpha_p \leq 1/2\),
\[
	\frac{K^2}{n} = \lambda^2 n^{2\alpha_p-1} \geq \tfrac12 \lambda^{(2p-1)/p}.
\]
In particular,
\begin{equation*}
    \mu_{n}^{0,h_n}(\varphi_0 \geq n + \lambda n^{\alpha_p})
    \leq
    c' e^{-\frac{\lambda^{(2p-1)/p}}{8\beta}}.
\end{equation*}
where \(c'>0\) depends only on \(p,\beta\). Combining this with the stochastic ordering described at the beginning of the section and Lemma~\ref{lem:Gaussian:height_normal_dev_proba_UB}, we get the next lemma.
\begin{lemma}
    \label{lem:Gaussian:height_dev_proba_UB}
    Let \(p\geq 1,\beta>0\). Let \(S\) be the random walk with i.i.d.\ steps of law \(\calN(0,\beta)\). There are \(c>0, n_0\geq 1,\lambda_0\geq 0\) such that, for any \(\lambda\geq \lambda_0\) and \(n\geq n_0\),
    \begin{equation*}
        P\bigl(S_0 \geq n + \lambda n^{\alpha_p} \given S_{-n}=S_n =0,\, S\geq h_n\bigr) \leq
        e^{-c \lambda^{(2p-1)/p}}.
    \end{equation*}
\end{lemma}

\appendix

\section{Gaussian random walk}
\label{app:Gaussian_RW}

In this appendix, we collect some classical results on Gaussian random walks that are used in Section~\ref{sec:Gaussian_p}. Below, \((X_k)_{k\geq 1}\) denotes an i.i.d.\ sequence of \(\calN(0,\beta)\) random variables, \(S_0\) is an \(\calN(0,1)\) random variable independent of \((X_k)_{k\geq 1}\), and
\begin{equation*}
    S_{n} = S_0 + \sum_{k=1}^n X_k.
\end{equation*}

\subsection{Gaussian process}

\begin{lemma}
    \label{lem:Gaussian_Walk:Walk_Bridge_are_Gaussian}
    For any \(n\geq 1\), under \(P(\cdot \given S_0=0)\), the vector \((S_0,S_1,\dots,S_n)\) is a centred Gaussian vector with covariances
    \begin{equation*}
        E(S_iS_j\given S_0=0) = \beta i,\quad 0\leq i \leq j \leq n.
    \end{equation*}
    Moreover, for \(n\geq 2\), under \(P(\cdot \given S_0 =0 = S_n )\), the vector \((S_0,S_1,\dots,S_n)\) is a centred Gaussian vector with covariances
    \begin{equation*}
        E(S_iS_j\given S_0 =0=S_n) = \beta \frac{i(n-j)}{n},\quad 0\leq i \leq j \leq n.
    \end{equation*}
\end{lemma}
\begin{proof}
    The first part is trivial. The second part follows from the first and the formula for conditional densities of Gaussian vectors.
\end{proof}

\subsection{Comparison with Brownian excursion}
\label{app:comparison_Brown_excursion}

Recall that the Brownian excursion of length \(L\) is a \(\R\)-valued process \((\BrownExc_t)_{t\in [0,L]}\) with
\begin{itemize}
    \item continuous sample paths;
    \item \(\BrownExc_0=\BrownExc_L=0\), \(\BrownExc_t>0\) for \(t\in (0,L)\);
    \item marginals given by: for \(m\geq 1\), \(0<t_1<\dots<t_m<L\), \((\BrownExc_{t_1},\dots,\BrownExc_{t_m})\) has a density with respect to Lebesgue given by (see for example~\cite{Takacs-1991})
    \begin{multline}
    \label{eq:def:Brownian_excursion_density}
        g_{t_1,\dots,t_m}^{(L)}(x_1,\dots, x_m)
        \\=
        2\sqrt{2 \pi L^3} q_{t_1}(x_1)p_{t_2-t_1}(x_1,x_2)\dots p_{t_m-t_{m-1}}(x_{m-1},x_m) q_{L-t_m}(x_m),
    \end{multline}
    where
    \begin{equation*}
        q_t(x) = \mathds{1}_{x>0} \frac{x}{\sqrt{2\pi t^3}}e^{-x^2/2t},
        \quad
        p_t(x,y) = \mathds{1}_{y>0}\frac{1}{\sqrt{2\pi t}}\bigl(e^{-(x-y)^2/2t}-e^{-(x+y)^2/2t}\bigr).
    \end{equation*}
\end{itemize}

The next lemma is a direct generalisation of the observation that the Brownian excursion of length \(n\) taken at integer times stochastically dominates the (suitably scaled) Gaussian random walk excursion.
\begin{lemma}
    \label{lem:Gaussian:excursion_comparison}
    Let \(n\geq 1\), \(T=1/\beta>0\). Let \(f_1,\dots, f_{n-1}:(0,+\infty)\to [0,+\infty)\) be measurable functions such that
    \begin{gather*}
        E\bigl( f_1(S_1)\cdots f_{n-1}(S_{n-1})\bgiven S_0=S_n=0, S\geq 0\bigr)<\infty,
        \\
        E\bigl( f_1(\sqrt{\beta}\BrownExc_1)\cdots f_{n-1}(\sqrt{\beta}\BrownExc_{n-1})\bigr)<\infty.
    \end{gather*}
    Let \(\mu\) be the probability measure on \(\R^{\{0,\dots, n\}}\) given by
    \begin{equation*}
        d\mu(x_0,\dots,x_n)
        \propto
        \delta(x_0)\delta(x_n)\Bigl(\prod_{i=1}^{n-1}f_i(x_i) \mathds{1}_{x_i>0} \, dx_i\Bigr) e^{-T\sum_{i=1}^n(x_i-x_{i-1})^2},
    \end{equation*}
    and \(\nu\) be the probability measure on \(\R^{\{0,\dots, n\}}\) given by
    \begin{equation*}
        d\nu(y_0,\dots,y_n)
        \propto
        \delta(y_0)\delta(y_n)g_{1,\dots,n-1}^{(n)}\bigl(\sqrt{T}y_1,\dots, \sqrt{T}y_{n-1}\bigr) \Bigl(\prod_{i=1}^{n-1}f_i(y_i) \, dy_i\Bigr).
    \end{equation*}
    Then,
    \begin{equation*}
        \mu \preccurlyeq \nu.
    \end{equation*}
\end{lemma}
\begin{proof}
    Introduce
    \begin{equation*}
        H(x) = \frac{1}{2}\sum_{i=1}^n (x_i-x_{i-1})^2.
    \end{equation*}
    First note that
    \begin{equation*}
        g_{1,\dots,n-1}^{(n)}\bigl(\sqrt{T}y_1,\dots, \sqrt{T}y_{n-1}\bigr)
        =
        c_{n,\beta} y_1y_{n-1} e^{-TH(y)} \Bigl(\prod_{i=2}^{n-1}(1-e^{-2Tx_{i}x_{i-1}})\Bigr)\prod_{i=1}^{n-1}\mathds{1}_{y_i>0},
    \end{equation*}
    where \(c_{n,\beta}\) is a normalisation constant depending only on \(n,\beta\).
    We check that the Holley condition is satisfied. We need to check that, for any \(x,y\in [0,+\infty)^{n-1}\),
    \begin{multline*}
        g_{1,\dots,n-1}^{(n)}\bigl(\sqrt{T}(x\vee y)\bigr)\Bigl(\prod_{i=1}^{n-1}f_i\bigl((x\vee y)_i\bigr)\Bigr)
        \Bigl(\prod_{i=1}^{n-1}f_i\bigl((x\wedge y)_i\bigr) \mathds{1}_{(x\wedge y)_i>0} \Bigr) e^{-T H(x\wedge y)}
        \\\geq
        g_{1,\dots,n-1}^{(n)}\bigl(\sqrt{T}(y)\bigr)\Bigl(\prod_{i=1}^{n-1}f_i\bigl( y_i\bigr)\Bigr)
        \Bigl(\prod_{i=1}^{n-1}f_i(x_i) \mathds{1}_{x_i>0}\Bigr) e^{-TH(x)}.
    \end{multline*}
    As
    \begin{equation*}
        \prod_{i=1}^{n-1}\mathds{1}_{(x\vee y)_i >0}f_i\bigl((x\vee y)_i\bigr) \mathds{1}_{(x\wedge y)_i >0} f_i\bigl((x\wedge y)_i\bigr)
        =
        \prod_{i=1}^{n-1}f_i( y_i) \mathds{1}_{y_i >0} f_i(x_i)\mathds{1}_{x_i >0},
    \end{equation*}
    it is equivalent to check that
    \begin{multline*}
        (x_1\vee y_1)(x_{n-1}\vee y_{n-1})e^{-T H(x\vee y)} \Bigl(\prod_{i=2}^{n-1}(1-e^{-2T(x\vee y)_{i}(x\vee y)_{i-1}})\Bigr)  e^{-T H(x\wedge y)} 
        \\\geq
        y_1y_{n-1}e^{-T H( y)} \Bigl(\prod_{i=2}^{n-1}(1-e^{-2Ty_{i}y_{i-1}})\Bigr) e^{-TH(x)}.
    \end{multline*}
    This last inequality follows from observing that
    \begin{equation*}
        H(x\vee y) + H(x\wedge y) \leq H(x) + H(y),
    \end{equation*}
    and that the functions \(x\mapsto x\) and \(x\mapsto 1-e^{-x}\) are non-decreasing on \([0,+\infty)\).
\end{proof}

\subsection{Walk excursion marginals}

\begin{lemma}
    \label{lem:Gaussian:excursion_marginals}
    For any \(1\leq k\leq n/2\), \(a> 0\),
    \begin{equation*}
        P(S_k \geq a \given S_0=0 =S_n, S\geq 0)
        \leq
        \frac{4}{\sqrt{\pi}}\Bigl(\frac{a}{\sqrt{\beta k}} + \frac{\sqrt{\beta k}}{ a}\Bigr)e^{-a^2/2\beta k}.
    \end{equation*}
\end{lemma}
\begin{proof}
    Let \(\rmP^+ = P(\cdot \given S_0 = 0 = S_n, S_i> 0\; \forall i=1,\dots, n-1)\).
    Let \((\BrownExc_t)_{t\in [0,n]}\) be the standard Brownian excursion of length \(n\). From Lemma~\ref{lem:Gaussian:excursion_comparison},
    \begin{equation*}
        (S_0,S_1,\dots,S_n) \preccurlyeq \sqrt{\beta}(\BrownExc_0,\BrownExc_1,\dots,\BrownExc_n),
    \end{equation*}
    with \(S\sim \rmP^+\). Now, the density of \(\BrownExc_t\) with respect to Lebesgue is given by
    \begin{equation*}
        g_t(x) = \frac{2}{\sqrt{2\pi}} \Bigl(\frac{ n}{ (n-t) t}\Bigr)^{\frac{3}{2}} x^2 e^{-x^2/2t}e^{-x^2/2(n-t)}\mathds{1}_{x>0}.
    \end{equation*}
    In particular, for \(k\leq\frac{n}{2}\), and \(T=1/\beta\),
    \begin{multline*}
        \rmP^+(S_k\geq a)
        \leq
        \int_{\sqrt{T} a}^{\infty} g_k(x) dx
        \leq
        \frac{4}{\sqrt{\pi}} k^{-\frac{3}{2}} \int_{\sqrt{T} a}^{\infty} x^2 e^{-x^2/2k}dx
        \\
        =
        \frac{4}{\sqrt{\pi}} \int_{\frac{\sqrt{T} a}{\sqrt{k}}}^{\infty}x^2 e^{-x^2/2}dx
        \leq
        \frac{4}{\sqrt{\pi}}\Bigl(\frac{\sqrt{T} a}{\sqrt{k}} + \frac{\sqrt{k}}{\sqrt{T} a}\Bigr)e^{-Ta^2/2k}.
        \tag*{\(\qed\)}
    \end{multline*}
    \renewcommand{\qedsymbol}{}
\end{proof}

\bibliographystyle{plain}
\bibliography{BibTeX}

\end{document}